\documentclass[11pt]{article}
\usepackage{hyperref}

\usepackage[english]{babel}
\usepackage[left=0.5in,right=0.5in,top=1in,bottom=1.5in]{geometry}
\usepackage{amssymb,amsmath,amsthm,amscd,ifthen}
\usepackage{mathtools}
\mathtoolsset{showonlyrefs}
\usepackage{graphicx}
\usepackage{xcolor}
\definecolor{ultramarine}{rgb}{0.07, 0.04, 0.56}
\definecolor{darkspringgreen}{rgb}{0.09, 0.45, 0.27}
\hypersetup{colorlinks,linkcolor={darkspringgreen},citecolor={ultramarine},urlcolor={ultramarine}}
\usepackage{bbm}
\usepackage{dsfont}
\usepackage{comment}
\usepackage{framed}
\usepackage{transparent}
\usepackage{color}
\usepackage{dirtytalk}
\usepackage{authblk}
\usepackage{framed}
\usepackage[authoryear]{natbib}
\bibpunct{[}{]}{,}{n}{,}{,}
\usepackage[shortlabels]{enumitem}
\usepackage{soul} 

\newcommand{\ip}[2]{\ensuremath{\left\langle {#1}, {#2} \right\rangle}}

\newcommand{\what}{\ensuremath{\widehat{\omega}}}

\makeatletter
\def\BState{\State\hskip-\ALG@thistlm}
\makeatother

\newcommand{\R}{{\mathbb R}}

\DeclareMathOperator{\E}{\mathbb{E}}
\DeclareMathOperator{\ind}{\mathds{1}}

\newtheorem*{theorem*}{Theorem}

\DeclareMathOperator*{\argmin}{arg\,min}
\DeclareMathOperator*{\arginf}{arg\,inf}
\newcommand{\inr}[1]{\left\langle #1 \right\rangle}

\definecolor{cd60952}{RGB}{214,9,82}

\date{}

\newtheorem{Theorem}{Theorem}[section]
\newtheorem{Lemma}[Theorem]{Lemma}

\newtheorem{Proposition}[Theorem]{Proposition}

\newtheorem{Example}[Theorem]{Example}

\DeclareMathOperator{\Tr}{Tr}

\DeclarePairedDelimiter{\norm}{\lVert}{\rVert}
\DeclarePairedDelimiter{\abs}{\lvert}{\rvert}

\newcommand{\myparagraph}[1]{\paragraph*{#1.}}

\title{Suboptimality of Constrained Least Squares and Improvements via Non-Linear Predictors}
\author[1]{Tomas Va\v{s}kevi\v{c}ius}
\author[2]{Nikita Zhivotovskiy}
\affil[1]{Department of Statistics, University of Oxford, \href{mailto:tomas.vaskevicius@stats.ox.ac.uk}{tomas.vaskevicius@stats.ox.ac.uk}}
\affil[2]{Department of Mathematics, ETH Z\"{u}rich, \href{mailto:nikita.zhivotovskii@math.ethz.ch}{nikita.zhivotovskii@math.ethz.ch}}

\begin{document}

\maketitle

\abstract{
We study the problem of predicting as well as the best linear predictor in a bounded Euclidean ball with respect to the squared loss. When only boundedness of the data generating distribution is assumed, we establish that the least squares estimator constrained to a bounded Euclidean ball does not attain the classical $O(d/n)$ excess risk rate, where $d$ is the dimension of the covariates and $n$ is the number of samples. In particular, we construct a bounded distribution such that the constrained least squares estimator incurs an excess risk of order $\Omega(d^{3/2}/n)$ hence refuting a recent conjecture of Ohad Shamir [JMLR 2015]. In contrast, we observe that non-linear predictors can achieve the optimal rate $O(d/n)$ with no assumptions on the distribution of the covariates. We discuss additional distributional assumptions sufficient to guarantee an $O(d/n)$ excess risk rate for the least squares estimator. Among them are certain moment equivalence assumptions often used in the robust statistics literature. While such assumptions are central in the analysis of unbounded and heavy-tailed settings, our work indicates that in some cases, they also rule out unfavorable bounded distributions. }

\section{Introduction}

We study random design linear regression under boundedness assumptions on the data generating distribution. Let $S_{n}$ denote a sample of $n$ i.i.d.\ input-output pairs $(X_{i}, Y_{i}) \in \R^{d} \times \R$ sampled from some unknown distribution $P$. In a traditional statistical learning theory setup, the aim of a (linear) learning algorithm is to map the observed learning sample $S_{n}$ to a linear predictor $\langle \what, \cdot \rangle$ that incurs a small \emph{risk} $R(\what) = \E(Y - \langle \what, X\rangle)^2$, where the random pair $(X, Y)$ is distributed according to $P$. In this paper, we analyze the performance of least squares (or empirical risk minimization (ERM)) estimators constrained to bounded Euclidean balls.

As a motivating example, consider the well-specified model 
$Y = \langle \omega^{*}, X \rangle + \xi.$
Here $\omega^{*} \in \R^{d}$, $\xi$ is zero mean and independent of $X$; we always assume that $Y$ is square integrable and that the covariance matrix $\Sigma$ of $X$ exists. Assuming additionally that $n \geq 2d$, $\xi$ is Gaussian and that $X$ is zero mean multivariate Gaussian with invertible covariance matrix $\Sigma$, a basic result \citep[Theorem 1.1]{breiman1983many} implies that the \emph{excess risk} of unconstrained least squares $\what$  (also known as the ordinary least squares estimator) satisfies
\begin{equation}
  \label{eq:well-specified-excess-risk}
  \E R(\what) - R(\omega^{*}) \lesssim \frac{dR(\omega^{*})}{n},
\end{equation}
where the expectation is taken with respect to the sample $S_{n}$, the notation $\lesssim$ suppresses absolute multiplicative constants, and the optimal risk $R(\omega^{*})$ is equal to the variance of the noise random variable $\xi$.
Remarkably, the bound \eqref{eq:well-specified-excess-risk} depends neither on the exact form of the covariance matrix $\Sigma$ nor on the magnitude of $\omega^{*}$.
Recent work \citep[Theorem 1]{mourtada2019exact} shows that if the model is well-specified then for any distribution of the covariates $X$ such that the sample covariance matrix is almost surely invertible and any $n \ge d$, the excess risk of unconstrained least squares is \emph{exactly} equal to the minimax risk. While the above result attests to the existence of regimes where least squares is a statistically optimal estimator in a minimax sense, there is a growing interest in the statistics and machine learning communities in understanding the robustness of statistical estimators to  various forms of model \emph{misspecification}. For instance, the regression function $\E(Y \mid X = \cdot)$ might be non-linear, or the distribution of the noise random variable $\xi$ might depend on the corresponding covariate $X$. 

Many authors have matched the $O(d/n)$ rate \eqref{eq:well-specified-excess-risk} for ERM-based algorithms under significantly less restrictive assumptions than that of a well-specified model with Gaussian covariates (e.g., assuming a favorable covariance structure and sub-Gaussian noise \citep{hsu2014random}, assuming $L_{q}$--$L_2$ (for some $q > 2$) moment equivalence of the marginals $\langle \omega, X \rangle$ and the noise random variable $\xi$ \citep{audibert2011robust, oliveira2016lower, catoni2016pac, mourtada2019exact}, or the weaker small-ball assumption \citep{Mendelson:2015:LWC:2799630.2699439, lecue2016performance}). Moment equivalence type assumptions allow for modelling heavy-tailed distributions, and, in particular, they have played a crucial role in recent developments in the robust statistics literature (e.g., \citep{catoni2016pac, chinot2019robust, lugosi2016risk, mendelson2020robust}); however, in some cases, such assumptions only hold with constants that can deteriorate arbitrarily with respect to the parameters of the unknown distribution $P$, even for light-tailed or bounded distributions. For instance, the smallest constant with respect to which $\operatorname{Bernoulli}(p)$ distribution satisfies the $L_{4}$--$L_{2}$ moment equivalence can get arbitrarily large for small $p$. In the context of linear regression, the work \citep[a discussion following Proposition 4.8]{catoni2016pac} highlights that some of the prior results on the performance of least squares relying on such assumptions can have constants that may unintentionally depend on the dimension of the covariates $d$. Recent literature has further accentuated this problem and witnessed an emerging interest in refining moment equivalence and small-ball assumptions \citep{saumard2018optimality, chinot2019robust, mendelson2017extending}. 

While the moment equivalence assumptions allow us to study unbounded and possibly heavy-tailed problems, such assumptions might impose undesirable structural constraints on the unknown distribution $P$ and, in some cases, result in overly optimistic bounds, as our work suggests. In this work we take an alternative point of view, frequently adopted in the aggregation theory literature \citep{nemirovski2000topics, tsybakov2003optimal, juditsky2008learning, audibert2009fast, lecue2013empirical, rakhlin2017empirical}: we impose no assumptions on the distribution $P$ other than boundedness and aim to obtain a (possibly non-linear) predictor  that performs at least as well as the best linear predictor in  $\mathcal{W}_{b} = \{\omega \in \R^{d} : \|\omega\| \leq b \}$, where $\|\cdot\|$ denotes the Euclidean norm. We remark that some distributional assumptions need to be made since otherwise, any algorithm that returns a linear predictor (including least squares) can incur arbitrarily large excess risk (cf.\ the lower bounds in \citep{shamir2015sample, lecue2016performance, catoni2016pac}).

Let $\what_{b} = \what_{b}(S_{n})$ denote a \emph{proper} estimator which corresponds to a linear function $\langle \what_{b}, \cdot \rangle$ for some
$\what_{b} \in \mathcal{W}_{b}$.
Otherwise, the estimator is called \emph{improper}. Fix any proper estimator $\what_{b}$ and any constants $r, m > 0$. The recent work \citep{shamir2015sample} shows that there exists a distribution $P = P(\what_{b}, r, m)$ satisfying $\| X\| \leq r$ almost surely and $\|Y\|_{L_{\infty}(P)} \leq m$ such that the following lower bound holds for any large enough sample size $n$ (see Section~\ref{sec:corollaries} for a more general statement):
\begin{equation}
  \label{eq:shamir-simplified}
  \E R(\what_{b}) - \inf\limits_{\omega \in \mathcal{W}_{b}}R(\omega)
  \gtrsim \frac{dm^2}{n} + \frac{r^2b^2}{n}.
\end{equation}
Note that the first term in the above lower bound corresponds to the rate in the upper bound \eqref{eq:well-specified-excess-risk}: the excess risk of a best predictor in class $\mathcal{W}_{b}$ is upper bounded by that of a zero function, whose risk is in turn bounded by $m^{2}$. On the other hand, the second term in \eqref{eq:shamir-simplified} shows that in the absence of simplifying distributional assumptions, the statistical performance of linear predictors can deteriorate arbitrarily with respect to the boundedness constants $r,b$, even in one-dimensional settings; in contrast, the upper bound \eqref{eq:well-specified-excess-risk} does not depend on $b$ and $r$.

Imposing only boundedness constraints on $P$, we study excess risk bounds of least squares performed over the class $\mathcal{W}_{b}$, denoted in what follows by $\what_{b}^{\operatorname{ERM}}$. A baseline for our work is a conjecture proposed in \citep{shamir2015sample} postulating the statistical optimality of the constrained least squares estimator $\what^{\operatorname{ERM}}_{b}$ in a sense that it matches the lower bound \eqref{eq:shamir-simplified}. For some of the recent discussions and attempts to resolve this conjecture see, for example, the works \citep*{koren2015fast, balazs2016chaining, gonen2017average, Wang2018RevisitingDP}.
The existing results, however, only partially address this conjecture, 
restricting to the regimes where $br \sim m$ (the notation $a \sim b$ means $a \lesssim b \lesssim a$). Specifically, the best known guarantees that can be obtained, for instance, via localized Rademacher complexity arguments \citep{bartlett2005local, koltchinskii2006local, liang2015learning} yielding the following upper bound
\begin{equation}
  \label{eq:excess-risk-suboptimal-b-plus-m}
  \E R(\what^{\operatorname{ERM}}_{b}) - \inf\limits_{\omega \in \mathcal{W}_{b}}R(\omega)
  \lesssim
  \frac{dm^2}{n} + d\cdot\frac{r^2b^2}{n},
\end{equation}
We note that an overlooked aspect of the work \citep{shamir2015sample} is that the lower-bound \eqref{eq:shamir-simplified} proved for \emph{proper} algorithms is matched there via the \emph{improper} Vovk-Azoury-Warmuth (VAW) forecaster. Among the proper algorithms, least squares is arguably the most natural and most extensively studied candidate that could potentially match the lower bound \eqref{eq:shamir-simplified} (as conjetured by Shamir). Thus, a natural reformulation of Shamir's conjecture arises:
\begin{framed}
    Provided that the covariate vectors and the response variable are bounded almost surely, is the constrained least squares estimator $\what_{b}^{\operatorname{ERM}}$ optimal among all (potentially non-linear) estimators in a sense that it always matches the lower bound \eqref{eq:shamir-simplified}?
\end{framed}

We address this question by showing that there exist bounded distributions inducing a multiplicative $\sqrt{d}$ gap between the excess risk achievable by the constrained least squares estimator $\what_{b}^{\operatorname{ERM}}$ and that achievable via non-linear predictors. It is important to highlight that this statistical gap holds despite performing ERM over a \emph{convex} and \emph{bounded} function class with respect to the \emph{squared loss}, a setting considered to be favorable in the literature (see, e.g.,  \citep[Chapter 5]{koltchinskii2011oracle}). In particular, the so-called \emph{Bernstein class condition} (see\ \citep{bartlett2006empirical}) is always satisfied in our setup, which is known to imply fast rates for least squares in the bounded setup whenever the underlying class is not too complex. Our work identifies a contrasting scenario: we find that the least squares algorithm is suboptimal for a convex problem and as such, the failure of the least squares procedure cannot be attributed to complex/non-convex structure of the underlying class. Instead, we identify the localized multiplier process as the complexity measure giving rise to unfavorable distributions; see Sections~\ref{sec:upper-bounds} and \ref{sec:lower-bound} for an extended discussion. We now state an informal version of our main result.

\begin{theorem*}[An informal statement]
    There exists a distribution $P$ satisfying $\|X\| \le 1$ almost surely and $\|Y\|_{L_{\infty}} \le 1$ (i.e., $r=m=1$) such that for any large enough $d$, $b \sim \sqrt{d}$ and large enough $n$ the following holds:
    \begin{equation}
        \label{eq:lower-bound-erm-informal}
        \E R(\what_{b}^{\operatorname{ERM}}) - \inf\limits_{\omega \in \mathcal{W}_{b}}R(\omega)
        \gtrsim \frac{d^{3/2}}{n}
        \sim \sqrt{d} \cdot \left(\frac{dm^2}{n} + \frac{r^2b^2}{n}\right).
    \end{equation}
    At the same time, there exists a non-linear predictor $\widehat{f}(\cdot)$ such that for any boundedness constant $m > 0$, any distribution $P$ (with no assumptions on the distribution of the covariates) satisfying $\|Y\|_{L_{\infty}} \le m$, and any $d, n > 0$, the following holds:
    \begin{equation}
    \label{eq:upper-bound-vovk-informal}
        \E R(\widehat{f}) - \inf\limits_{\omega \in \mathbb{R}^d}R(\omega)
        \lesssim \frac{dm^{2}}{n}.
    \end{equation}
\end{theorem*}
In particular, the lower bound \eqref{eq:lower-bound-erm-informal} resolves the conjecture of Ohad Shamir on the optimality of $\what_{b}^{\operatorname{ERM}}$ in the negative. The distribution $P$ used to prove the lower bound relies on a mixture of dense and sparse covariate vectors, with the majority of the samples having dense covariates; intuitively, the sparse covariates have high statistical leverage which in turn forces least squares to overfit on a small subset of the observed data. 

The construction of such a distribution is guided by our upper bounds presented in Section~\ref{sec:upper-bounds}, where we prove a refined version of the upper bound \eqref{eq:excess-risk-suboptimal-b-plus-m} that allows us to get rid of the excess factor $d$ appearing in the second term of \eqref{eq:excess-risk-suboptimal-b-plus-m}, while replacing the first term with a quantity that appears in minimax lower bounds for the well-specified model. In Section~\ref{sec:corollaries}, we demonstrate that some additional assumptions, such as $L_{4}$--$L_{2}$ moment-equivalence of the marginals $\langle \omega, X \rangle$ and the noise, are enough to ensure that the constrained least squares estimator matches the lower bound \eqref{eq:shamir-simplified}. An emerging picture is that there exist scenarios (namely, when statistical leverage scores are suitably correlated with the noise as we discuss in Section~\ref{sec:lower-bound}) under which the performance of least squares is sensitive to the constants with respect to which such assumptions are satisfied. Consequently, unfavorable distributions exist even when both the covariates and the response variables are bounded. Our lower bound \eqref{eq:lower-bound-erm-informal} is proved by constructing one such distribution with the moment equivalence constants proportional to $\sqrt{d}$.  Rather than viewing our lower bound as an isolated case, we note that there exists a spectrum of bounded distributions, ranging from the ones that satisfy moment equivalence conditions with absolute constants, to the ill-behaved ones used to prove our theorem above.

On the positive side, our work unveils the potential statistical improvements offered by non-linear predictors. In the theorem statement above, $\widehat{f}$ denotes a modified Vovk-Azoury-Warmuth forecaster \citep{vovk2001competitive, azoury2001relative} introduced in \citep{forster2002relative}. In addition to the $\sqrt{d}$ gap between \eqref{eq:lower-bound-erm-informal} and \eqref{eq:upper-bound-vovk-informal}, the non-linear predictor $\widehat{f}$ surpasses the lower bound \eqref{eq:shamir-simplified} that holds for \emph{any} algorithm returning a linear predictor in $\mathcal{W}_{b}$. In particular, it removes the dependence on the boundedness constants $b$ and $r$ appearing in the lower bound \eqref{eq:shamir-simplified} completely removing \emph{any} dependence on the distribution of the covariates. 

Regarding the terminology, we note the word pairs \emph{proper} and \emph{linear}, as well as \emph{improper} and \emph{non-linear} are used synonymously in this work. Thus, the separation between proper and improper learning is not to be confused with the separation between linear and non-linear \emph{procedures} considered in other works. For example, the work \citep{donoho1998minimax} finds a statistical gap between minimax rates achievable by procedures that are linear and non-linear in the observed response variables. In contrast, in our case, the linearity/properness means that an estimator always selects a predictor of the form $x \mapsto \langle x, w\rangle$ for some $w \in \mathcal{W}_{b}$, while non-linear/improper predictor is allowed to select predictors that are not of the above form.

Finally, we highlight that the construction used to prove our main lower bound \eqref{eq:lower-bound-erm-informal} is specifically designed to simultaneously satisfy moment equivalence assumptions (on the noise and the design) with ill-behaved constants. The fact that non-linear estimators can always achieve the optimal $d/n$ rate shows that they are insensitive to moment equivalence constants. Thus, the observed statistical gap between proper and improper estimators is of interest to the line of work investigating  estimators (typically proper) robust to heavy-tailed data. Indeed, following the first appearance of the preprint of this work, this observation has motivated the design of a non-linear statistical estimator robust to heavy-tailed data that works without any assumptions on the distribution of the covariates \citep{mourtada2021distribution}.

\subsection{Related Work}
\label{sec:related-work}
\myparagraph{Linear regression} Ordinary least squares and its variations (e.g., ridge regression or constrained least squares) have been studied extensively in the literature, with most of the results primarily focusing on the upper bounds. Many variations of our problem were previously considered in the literature (e.g., fixed-design regression, distributional assumptions different from boundedness or performance metrics that differ from the excess risk); see \citep{shamir2015sample} for a detailed comparison of different setups. For comprehensive surveys of existing work, we refer to \citep{audibert2010linear, audibert2011robust, hsu2014random, mourtada2019exact} and the books \citep{gyorfi2002distribution, koltchinskii2011oracle, wainwright2019high}. The lower bound \eqref{eq:shamir-simplified} is due to \citep{shamir2015sample}, and it is the tightest lower bound in the literature that covers the setting considered in our work. The best known upper bound in such a setting is of the form given in \eqref{eq:excess-risk-suboptimal-b-plus-m}; the gap between \eqref{eq:shamir-simplified} and \eqref{eq:excess-risk-suboptimal-b-plus-m} is currently not addressed in the literature.

Many of the existing upper bounds in the literature hold with high probability. In contrast, we focus on establishing suboptimality of constrained least squares and demonstrating a form of statistical separation between proper and improper estimators; thus, we concentrate on in-expectation analysis to convey our main findings without introducing additional technicalities. In the bounded regime, our upper bound for constrained least squares can be translated into high-probability results via standard arguments based on Talagrand's concentration inequality for empirical processes (see discussions in the book \citep{koltchinskii2011oracle}). At the same time, our upper bounds for ridge regression and Vovk-Azoury-Warmuth estimators are based on stability and online-to-batch arguments, respectively, neither of which easily generalizes to a high-probability counterpart.

\myparagraph{(Sub)optimality of ERM} Understanding statistical guarantees pertaining to estimators based on ERM has been a subject of intense study in many contexts. Among the simplest problems where ERM is known to incur suboptimal excess risk rates is the \emph{model selection aggregation} \citep{nemirovski2000topics, tsybakov2003optimal}, where the aim is to predict as well as the best function in a given finite class of bounded functions of size $M$. It is well-known that the \emph{non-convex} structure of finite classes inhibits the excess risk achievable by any proper estimator (including ERM over the class); in particular, proper estimators can only achieve a $\sqrt{\log M/{n}}$ rate as opposed to the optimal $\log M/n$ rate achievable by improper estimators \citep{juditsky2008learning}. Among the similarities between the aggregation setup and our work is that optimal rates are often achieved via procedures taking their roots in the sequential prediction setup (in our case, the Vovk-Azoury-Warmuth forecaster) \citep{audibert2009fast}. A key difference, however, is the fact that our results establish suboptimality of ERM for a \emph{convex} and bounded class, albeit with respect to the dependence on $d$ instead of $n$.

Constrained ERM with the squared loss is also actively studied in an on-going line of work concerning the shape restricted regression literature (e.g., \citep{chatterjee2014new, chatterjee2015risk, bellec2018sharp, han2019isotonic}), where the least squares projection is performed over constraint sets that may be significantly more complex than Euclidean balls. In particular, when considering some expressive nonparametric classes of functions, ERM can be either optimal \citep{han2019isotonic} or suboptimal \citep{birge1993rates, pmlr-v125-kur20a}, depending on some additional properties of these classes. In contrast, our results establish suboptimality of the constrained least squares estimator for a parametric class that has a small intrinsic complexity. The work \citep{chatterjee2014new} allows more general convex constraints and shows that ERM can be rate suboptimal. However, establishing suboptimality of ERM in our setting is more complicated: we are not free to choose an arbitrary ill-behaved convex constraint set and also, we study a random design setting and thus cannot choose a fixed set of ill-behaved covariates $X_i$. We additionally refer to \citep{birge2006model} for an extensive discussion on optimality and suboptimality of least squares and maximum likelihood estimators in different setups. 

A phenomenon separating statistical performance achievable by proper and improper estimators, related to the one observed in our work for linear regression, has recently attracted considerable attention in the logistic regression literature. Consider the setting of online logistic regression over a bounded Euclidean ball $\mathcal{W}_{b}$ in $\mathbb{R}^d$ and denote the number of rounds by $t$ (for the background on the setup see \citep{pmlr-v75-foster18a}). In this case, the cumulative regret of the online Newton step algorithm \citep{hazan2007logarithmic} is of order $e^{b}d\log t$. The exponential dependence on $b$ lead to a question formulated in \citep{mcmahan2012open} asking whether there exists an algorithm with logarithmic regret but polynomial dependence on the radius $b$ of the constraint set. The work \citep{hazan2014logistic} shows that such an algorithm does not exist in the class of proper estimators; the work \citep{pmlr-v75-foster18a} provides an improper algorithm that attains a cumulative regret guarantee of order $d\log(bt)$ with a doubly-exponential improvement in the dependence on $b$.

\subsection{Summary of our Techniques and Results}
As mentioned above, the upper bound \eqref{eq:excess-risk-suboptimal-b-plus-m} can also be readily obtained via the classical localized Rademacher complexity arguments \citep*{bartlett2005local, koltchinskii2006local} or the more recently introduced offset Rademacher complexity \citep*{liang2015learning}. Crucially, the lower bound \eqref{eq:shamir-simplified} and the upper bound \eqref{eq:excess-risk-suboptimal-b-plus-m} differ by a factor of $d$ in the worst case. It appears that the suboptimal dependence on the boundedness constants $br$ and $m$ arises due to an application of the Ledoux-Talagrand contraction inequality \citep{ledoux1991probability}. In particular, when $\norm{Y}_{L_{\infty}(P_{r})} \leq m$, the quadratic loss is $2(br + m)$-Lipschitz on $\mathcal{W}_{b}$, and the constant $(br + m)^{2}$ propagates into the resulting upper bounds. It is well-known that in the context of unbounded distributions, the application of the contraction argument can yield suboptimal bounds \citep{Mendelson:2015:LWC:2799630.2699439}; our work shows that the same is true in the classical bounded setup. In order to avoid the contraction step, we base our analysis on two components. First, we reduce our proof to the analysis of the \emph{localized multiplier and quadratic processes} as in \citep{lecue2013learning, Mendelson:2015:LWC:2799630.2699439}. Second, we use a version of Rudelson's inequality \citep{rudelson1999random} for sums of rank-one operators due to Oliveira \citep{oliveira2010sums} to analyze the localized quadratic process. We believe that our approach can also be used in the case of unbounded distributions: there is a version of Oliveira's bound for unbounded matrices \citep{klochkov2020uniform} that can be used to replace the assumption $\|X\| \le r$ with a strictly weaker sub-Gaussian tail assumption on the norms $\|X\|$; this could be seen as a step towards incorporating unbounded distributions within our framework while not relying on moment equivalence assumptions discussed above.

Using the notion of average stability, we prove a tight excess risk upper bound for the ridge regression estimator (see \citep{shalev2014understanding, koren2015fast, gonen2017average} for a detailed account of stability in our context). The novel ingredient in our proof is the exploitation of the curvature of the squared loss in the \emph{stability-fitting trade-off}. As a result, we show that the ridge estimator does not suffer from an excess factor $\log (\min\{n, d\})$ that appears in an upper bound on the localized quadratic process. Moreover, we demonstrate that this logarithmic term is unavoidable for least squares in some regimes, thus showing an interesting performance gap between constrained and penalized least squares.

Finally, the proof of our main lower bound, presented in Section~\ref{sec:lower-bound}, relies on a combination of some delicate exact computations and multiple applications of the matrix Chernoff and Bernstein inequalities \citep{tropp15}. Technical difficulties aside, the main challenge in proving our lower bound is constructing the example distribution used to establish that constrained least squares is suboptimal. In addition to the restrictions imposed by the boundedness constraints, we discuss other distributional assumptions sufficient to ensure that constrained least squares matches the rate \eqref{eq:shamir-simplified} (see \ Section~\ref{sec:corollaries}), thus making our main lower bound impossible. The intuition behind our construction is rooted in the form taken by our upper bounds. In particular, we construct a distribution tailored to make the localized multiplier process ill-behaved by simultaneously violating moment equivalence assumptions on the noise and statistical leverage score distributions. See Section~\ref{sec:lower-bound} for more details.

Below is a summary of our main contributions.
\begin{enumerate}
\item
  In Theorem~\ref{thm:main_thm}, we prove a tight upper bound on the expected excess risk of any constrained least squares estimator $\what_{b}^{\operatorname{ERM}}$. We demonstrate that the localized multiplier process is equal to the correlation between the squared noise and the statistical leverage scores of the covariates, a quantity that appears in minimax lower bounds in the well-specified case with independent noise \citep{mourtada2019exact}. In Proposition~\ref{prop:lowerbound}, we construct a distribution for which there exists a least squares solution $\what_{b}^{\operatorname{ERM}}$ whose excess risk is lower-bounded by the localized quadratic process.  Hence, both terms appearing in our upper bound are not improvable in general.

\item
   In Theorem~\ref{thm:main_thm_second}, we prove a tight excess risk upper bound for the ridge regression estimator. We recover the excess risk upper bound proved for constrained least squares, with the localized quadratic term replaced by a bias term that yields a logarithmic improvement.

\item
  Section~\ref{sec:corollaries} is dedicated to corollaries of our upper bounds. We show that under some assumptions frequently considered in the literature, the localized multiplier process is, up to log factors, upper-bounded by a term of order $dm^{2}/n$.  Consequently, in such regimes, the constrained least squares and the ridge regression estimators match the rate of the lower bound \eqref{eq:shamir-simplified} up to logarithmic factors.

\item
  In Theorem~\ref{thm:erm-ridge-lower-bound}, we construct a distribution $P$ with $r=1$ and $m=1$ such that for $b \sim \sqrt{d}$, any constrained least squares estimator $\what_{b}^{\operatorname{ERM}}$ incurs an excess risk that is larger by a $\sqrt{d}$ factor than the lower bound \eqref{eq:shamir-simplified} proved in \citep{shamir2015sample}. In particular, we refute the conjecture of Ohad Shamir on the optimality of the constrained least squares estimator.  

\item  The lower bound \eqref{eq:shamir-simplified} that only holds for proper \emph{linear} predictors was matched in \citep{shamir2015sample} via the \emph{non-linear} Vovk-Azoury-Warmuth (VAW) forecaster without tuning the regularization parameter. We observe that once the regularization parameter is tuned, the VAW forecaster yields an exponential improvement on the parameters $b$ and $r$ in the lower bound \eqref{eq:shamir-simplified}; in particular, our observation demonstrates a fundamental gap between the performance achievable by proper and improper estimators. In addition, we discuss a modified version of the VAW forecaster due to \citep{forster2002relative} that completely removes the dependence on the boundedness constants $b$ and $r$ in its excess risk upper bound.
\end{enumerate}

We present the proofs of the main results in Section~\ref{sec:proofs} with some of the details delegated to the appendix.

\subsection{Notation}
The subscript $r$ in $P_{r}$ denotes that the distribution $P_{r}$ of the random pair $(X, Y)$ that satisfies $\|X\| \leq r$ almost surely. The boundedness constant $m$ is an upper bound on the $L_{\infty}$ norm of the response variable $Y$. We denote Euclidean balls with radius $b$ by $\mathcal{W}_{b} = \{\omega \in \R^{d} : \|\omega\| \le b\}$. Let $\what_{b}^{\operatorname{ERM}}$ denote any ERM over $\mathcal{W}_{b}$, that is, any solution to
\begin{equation}
    \label{eq:erm}
    \what^{\operatorname{ERM}}_{b} \in \argmin\limits_{\omega \in \mathcal{W}_{b}}
  \sum\limits_{i = 1}^n(Y_i - \langle \omega, X_i\rangle)^2.    
\end{equation}
For any $\lambda \geq 0$, denote the \emph{regularized sample second moment matrix} by
\begin{equation}
  \label{eq:samplecov}
  \widehat{\Sigma}_{\lambda} = \frac{1}{n}\left(\lambda I_d + \sum\limits_{k = 1}^{n - 1}X_kX_k^\mathsf{T} + XX^\mathsf{T}\right),
\end{equation}
where we write $X$ instead of $X_n$ in order to simplify the notation in our main results, and correspondingly, we write $Y$ and $\xi$ instead of $Y_{n}$ and $\xi_{n}$. Given a regularization parameter $\lambda > 0$, the \emph{ridge estimator} is defined as
\begin{equation}
  \label{eq:ridge}
  \widehat{\omega}_{\lambda} = \argmin\limits_{\omega \in \mathbb{R}^d}
  \sum\limits_{i = 1}^n(Y_i - \langle \omega, X_i\rangle)^2 + \lambda\|\omega\|^2
  = \left(n\widehat{\Sigma}_{\lambda}\right)^{-1}\left(\sum_{i=1}^{n} Y_{i}X_{i} \right).
\end{equation}
Further, let $\omega^{*}_{b}$ denote any solution minimizing the population risk $R(\cdot)$ over $\mathcal{W}_{b}$ and let $\xi$ denote the \emph{noise variable}:
\begin{equation}
  \label{eq:omegaandxi}
  \omega^{*}_{b} \in \argmin_{\omega \in \mathcal{W}_{b}} R(\omega)
  \quad \text{and} \quad \xi = \xi_{b}(X,Y) = Y - \langle \omega^{*}_{b}, X\rangle,
\end{equation}
We denote positive numerical constants by $c, c_1, \ldots$ and note that their values may change from line to line; $a \lesssim b$ denotes the existence of a numerical constant $c$ such that $a \leq cb$; $a \sim b$ is a shorthand for $b \lesssim a \lesssim b$.  The notation $\norm{\cdot}$ denotes the Euclidean norm for vectors and the operator norm for matrices. For any $p \in [1, \infty]$, $\norm{\cdot}_{L_p}$ denotes the $L_{p}(P)$ norm, where the distribution $P$ will always be clear from the context.  With a slight abuse of notation, for any $\omega \in \R^{d}$, we let $\norm{\omega}_{L_{2}}^{2} = \norm{\ip{\omega}{X}}_{L_{2}}^{2} = \E \inr{\omega, X}^2$.  The $d \times d$ identity matrix is denoted by $I_d$ and $\text{Diag}(a_1, \ldots, a_d)$ denotes the diagonal matrix formed by $a_{1}, \ldots, a_d$. Finally, the indicator function of an event $E$ is denoted by $\ind_{E}$.

\section{Upper Bounds}
\label{sec:upper-bounds}

In this section, we provide two upper bounds: the first is for constrained least squares and the second is for the ridge estimator. These bounds will later motivate our construction separating the performance of least squares and non-linear predictors under the boundedness assumption.
\subsection{Performance of Constrained Least Squares}

Our first theorem is an upper bound on the excess risk of any constrained least squares estimator
$\what_{b}^{\operatorname{ERM}}$. The proof is deferred to Section~\ref{section:erm-proof}. We remark that the below upper bound is non-asymptotic, in contrast to the agnostic upper bound proved in \citep[Theorem 2.1]{audibert2011robust}. Also, we make no restrictions on the sample size $n$ and the uniqueness of least squares, as opposed to results that hold for the unconstrained least squares estimator (e.g., \citep[Proposition 1]{mourtada2019exact}).

\begin{Theorem}
    \label{thm:main_thm}
    For any $n,d,b, r > 0$ and any distribution $P_{r}$ satisfying $\E Y^2 < \infty$ it holds that
    \begin{equation}
        \label{eq:thm_main}
        \E R(\what^{\operatorname{ERM}}_{b}) - R(\omega^*_b) \lesssim \inf\limits_{\lambda > 0}\left(\frac{\E\xi^2X^\mathsf{T}\widehat{\Sigma}_{\lambda r^2}^{-1}X}{n} + \frac{\lambda r^2b^2}{n}\right) + \frac{r^2b^2\log(\min\{n, d\})}{n},
    \end{equation}
    where  $\what^{\operatorname{ERM}}_{b}$, $\widehat{\Sigma}_{\lambda r^2}$ and $\omega^{*}_{b}, \xi$ are defined in 
    \eqref{eq:erm}, \eqref{eq:samplecov} and \eqref{eq:omegaandxi} respectively.
\end{Theorem}

We comment on the structure of the above bound. Assume for the sake of presentation that $\widehat{\Sigma}_{0}$ is invertible. Then, with the choice $\lambda = 0$, we may rewrite the above upper bound as follows:
\begin{equation}
    \label{eq:quadratic-multiplier-explanation}
    \E R(\what^{\operatorname{ERM}}_{b}) -  R(\omega^*_b) \lesssim  \underbrace{\frac{\E\xi^2X^\mathsf{T}\widehat{\Sigma}_{0}^{-1}X}{n}}_{\text{Interaction with the noise}} + \underbrace{\frac{r^2b^2\log(\min\{n, d\})}{n}}_{\text{Low-noise complexity}}.
\end{equation}
The first term, which arises from the supremum of the localized multiplier process, shows the correlation between the noise $\xi$ and the \emph{statistical leverage score} $X^\mathsf{T}\widehat{\Sigma}_{0}^{-1}X$ (let $\mathbf{X} \in \R^{n \times d}$ denote the matrix with the $i$-th row equal to $X_{i}$; we may write $X^\mathsf{T}\widehat{\Sigma}_{0}^{-1}X = H_{nn}$, where $H \in \R^{n \times n}$ is the ``hat matrix'' defined as $H = \mathbf{X}(\mathbf{X}^{\mathsf{T}}\mathbf{X})^{-1}\mathbf{X}^{\mathsf{T}}$). If $d \le n$ and the noise random variable $\xi$ is independent of $X$, then the first term in \eqref{eq:quadratic-multiplier-explanation} corresponds essentially to the minimax optimal rate for unconstrained least squares regression \citep[Theorem 2]{mourtada2019exact} and is hence unimprovable in general.
The second term in \eqref{eq:quadratic-multiplier-explanation}, which arises from the supremum of the localized quadratic process, intuitively captures the problem complexity in low-noise regimes, that is, when $\xi$ is relatively small. In Proposition~\ref{prop:lowerbound}, we demonstrate a \emph{noiseless} problem such that for \emph{some} constrained least squares solutions the second term in \eqref{eq:quadratic-multiplier-explanation} is tight.

\subsection{Performance of the Ridge Regression Estimator}
We now turn to our second result, which provides an excess risk upper bound for the ridge regression estimator. The proof is deferred to Appendix~\ref{sec:secondmainthm}.
\begin{Theorem}
    \label{thm:main_thm_second}
    For any $n,d,b,r > 0$, any distribution $P_{r}$ satisfying $\E Y^2 < \infty$, and any choice of
    the regularization parameter $\lambda \gtrsim r^{2}$, it holds that
    \begin{equation}
        \label{eq:thm_main_second}
        \E R(\widehat{\omega}_{\lambda}) - R(\omega^*_b) \lesssim \frac{\E\xi^2X^\mathsf{T}\widehat{\Sigma}_{\lambda}^{-1}X}{n} + \frac{\lambda b^2}{n},
    \end{equation}
    where  $\widehat{\omega}_{\lambda}$, $\widehat{\Sigma}_{\lambda}$ and $\omega^{*}_{b}, \xi$ are defined in 
    \eqref{eq:ridge}, \eqref{eq:samplecov} and \eqref{eq:omegaandxi} respectively.
\end{Theorem}
We remark that one may not choose an arbitrary small value of $\lambda$ and hence the above theorem does not directly imply the result of Theorem~\ref{thm:main_thm}. Also, note that the empirical risk functional is not normalized in our work, and hence the above choice of $\lambda$ corresponds to the regularization parameter scaling as $r^{2}/n$ for normalized empirical risks considered in some other works.

\subsection{Discussion on the Optimality of Theorems~\ref{thm:main_thm}~and~\ref{thm:main_thm_second}}

At the first sight, the upper bounds presented in Theorems~\ref{thm:main_thm}~and~\ref{thm:main_thm_second} look similar; however, there are several important differences that we emphasize below:
\begin{itemize}
    \item The ridge estimator $\widehat{\omega}_{\lambda}$ does not necessarily belong to the set $\mathcal{W}_{b}$ and hence it is an \emph{improper} estimator, in contrast to the least squares estimator $\what_{b}^{\operatorname{ERM}}$. In addition, a single ridge regression estimator $\what_{\lambda}$ provides a \emph{family} of upper bounds, one for each choice of $b$, whereas the least squares estimator $\what_{b}^{\operatorname{ERM}}$ itself depends on the choice of $b$.
    
    \item The parameter $\lambda$ in Theorem \ref{thm:main_thm} is used to optimize the trade-off between the two terms and does not affect the estimator itself. In Section~\ref{sec:corollaries}, we demonstrate how the flexibility to optimize $\lambda$ in the upper bound in Theorem~\ref{thm:main_thm} allows to match the ``slow rate'' term in a more general version of the lower bound \eqref{eq:shamir-simplified}. It is not immediately evident whether the same is true for the ridge estimator $\what_{\lambda}$ in view of Theorem~\ref{thm:main_thm_second}.
    
    \item Theorem~\ref{thm:main_thm} contains an extra factor $\log(\min \{n, d\})$ that is not present in  Theorem~\ref{thm:main_thm_second}. In Proposition~\ref{prop:lowerbound} below, we show that this logarithmic factor is inherent for constrained least squares.
    
    \item The analysis of the constrained least squares estimator is based on the empirical process theory and concentration inequalities for random matrices; the analysis of the ridge estimator is based on an average stability argument.
\end{itemize}

The next result, proved in Appendix~\ref{sec:prooflowerbound}, shows that the extra logarithmic factor that appears in Theorem~\ref{thm:main_thm} but not in Theorem~\ref{thm:main_thm_second} is unavoidable. Our proof technique is based on an instance of the \emph{coupon collector problem}, a common tool for establishing that some logarithmic factors are unimprovable in the noise-free binary classification problem (see, e.g., \citep*{Bousquet20} and the references therein). We remark that the below lower bound holds for \emph{some} ERM, yet there might exist other ERMs which may violate the below lower bound.
\begin{Proposition}
\label{prop:lowerbound}
    For any large enough sample size $n$, any $d \geq n$, and any $r,b > 0$, there exists a distribution $P_r = P_r(n, d, b)$
    with $\xi = 0$ such that the following lower bound holds for some constrained least squares estimator:
    \[
        \E R(\widehat{\omega}^{\operatorname{ERM}}_{b}) - R(\omega^*_b) \gtrsim \frac{r^2b^2\log n}{n}.
    \]
\end{Proposition}

We now discuss some closely related lower bounds indicating that Theorem \ref{thm:main_thm} and Theorem \ref{thm:main_thm_second} cannot be improved in a certain sense. First, \citep[Theorem 2]{mourtada2019exact} shows that if $d \le n$, then for any distribution of the covariates $X$ such that the sample covariance matrix is invertible almost surely and any linear predictor $\tilde{\omega}$, there is a joint distribution $(X, Y)$ with independent zero mean Gaussian noise $\xi$ such that the following holds:
\begin{equation}
\label{eq:mourtlowerbound}
\E R(\tilde{\omega}) - R(\omega^*_{\infty}) \ge \E\xi^2\left(\E\frac{X^{\mathsf{T}}\widehat{\Sigma}_{0}^{-1}X}{n - X^{\mathsf{T}}\widehat{\Sigma}_{0}^{-1}X}\right)= \E\left(\frac{\xi^2X^{\mathsf{T}}\widehat{\Sigma}_{0}^{-1}X}{n - X^{\mathsf{T}}\widehat{\Sigma}_{0}^{-1}X}\right),
\end{equation}
where $\omega^*_{\infty}$ minimizes the risk among all vectors in $\mathbb{R}^d$. This term is an exact analog of the first term in the upper bounds \eqref{eq:thm_main} whenever $\frac{1}{n}X^{\mathsf{T}}\widehat{\Sigma}_{0}^{-1}X$ is separated from $1$ and $\lambda = 0$. Second, it is shown in \citep[Theorem 3]{shamir2015sample} that for $d = 1$, any $m,b,r$ satisfying $br \geq 2m$, and any estimator $\tilde{\omega}$ taking its value in $\mathcal{W}_{b}$, there exists a distribution $P_{r}$ such that
$\norm{Y}_{L_{\infty}} \leq m$ and 
\begin{equation}
\label{eq:shamirlowerbound}
    \E R(\tilde{\omega}) - R(\omega^*_{b}) \gtrsim \min\left\{m^2, \frac{r^{2}b^2}{n}\right\}.
\end{equation}
Proposition \ref{prop:lowerbound} and the lower bounds \eqref{eq:mourtlowerbound}, \eqref{eq:shamirlowerbound} indicate the existence of regimes such that none of the terms appearing in Theorems~\ref{thm:main_thm}~and~\ref{thm:main_thm_second} can be improved in general; the full picture is, however, more subtle. In particular, for constrained least squares $\what_{b}^{\operatorname{ERM}}$ (i.e., $b < \infty$) we obtain a non-trivial upper bound even in the regimes when the leverage scores $\frac{1}{n}X^{\mathsf{T}}\widehat{\Sigma}_{0}^{-1}X$ are close to $1$. On the other hand, the global bound (i.e., $b = \infty$) stated in \eqref{eq:mourtlowerbound} can deteriorate if the leverage scores are close to $1$. We also remark that the construction of distributions used to prove the lower bound \eqref{eq:shamirlowerbound} rely on non-zero noise problems, in contrast to the construction of the distribution used to prove Proposition~\ref{prop:lowerbound} in our work. Therefore, the term $\frac{r^{2}b^2}{n}$ that appears in the lower bound \eqref{eq:shamir-simplified} is not directly related to the problem complexity in the low-noise regimes, as opposed to the second term in Theorem~\ref{thm:main_thm}. 

\subsection{Bounds on the Multiplier Term}
\label{sec:corollaries}

The upper bounds presented in Theorems~\ref{thm:main_thm} and \ref{thm:main_thm_second} hold assuming boundedness of the covariates $\norm{X} \leq r$, and square integrability of the labels $\E Y^{2} < \infty$. In this section, we turn to the bounded setting when in addition it holds that $\norm{Y}_{L_{\infty}} \leq m$ . Note that the terms $r^{2}b^{2}\log(\min\{n,d\})/n$ and $r^{2}b^{2}/n$ appearing in Theorems~\ref{thm:main_thm} and \ref{eq:thm_main_second} respectively match the corresponding term (up to the logarithmic factor) that appears in the lower bound~\eqref{eq:shamir-simplified}. As a result, the constrained least squares and the ridge estimators can only exhibit suboptimal behaviour when the \emph{multiplier term} $\E\xi^2X^\mathsf{T}\widehat{\Sigma}_{\lambda}^{-1}X/n$ is significantly larger than $dm^{2}/n$ (cf.\ the discussion following Theorem~\ref{thm:main_thm}). In this section, we discuss several distributional assumptions that in addition to boundedness ensure a well-behaved multiplier term. Intuitively, all the assumptions considered below introduce a form of independence between the noise variable and the statistical leverage scores. In Section~\ref{sec:lower-bound}, we demonstrate that once such assumptions are violated, the multiplier term can be larger than $dm^{2}/n$ by a multiplicative $\sqrt{d}$ factor, despite the restriction to a family of bounded distributions.

The main result of the current section is Proposition~\ref{prop:l4l2}, which shows that the constrained least squares and the ridge estimators match the lower bound \eqref{eq:shamir-simplified} under $L_{4}$--$L_{2}$ moment equivalence assumptions specified below. We remark, however, that tighter lower bounds than that of \eqref{eq:shamir-simplified} might be possible under some of the assumptions considered below. In particular, it is well-known that assumptions closely related to moment equivalence considered below (e.g., small-ball \citep{Mendelson:2015:LWC:2799630.2699439}) might simplify the quadratic process, which gives rise to the $r^{2}d^{2}\log(\min\{n, d\})/n$ term in Theorem~\ref{thm:main_thm}. However, as discussed above, the multiplier term is responsible for the suboptimality of constrained least squares and in this section we are mainly trying to understand the conditions sufficient to ensure a well-behaved multiplier term.

Before discussing the upper bounds on the multiplier term, let us briefly review a more general version of the lower bound \eqref{eq:shamir-simplified} that holds without any restrictions on the sample size $n$:
\begin{equation}
    \label{eq:shamir-full}
    \E R(\tilde{\omega}) - \inf\limits_{\omega \in \mathcal{W}_{b}}R(\omega)
      \gtrsim
      \min\left\{m^2, \min\left\{\frac{dm^2}{n}, \frac{rbm}{\sqrt{n}}\right\} +
      \frac{r^2b^2}{n}\right\},
\end{equation}
where $\tilde{\omega}$ is any linear predictor in the set $\mathcal{W}_{b}$\footnote{
The statement of Theorem 1 in \citep{shamir2015sample} allows for $\what \in \mathbb{R}^d$ instead of $\what \in \mathcal{W}_{b}$. However, Lemma~2 in \citep{shamir2015sample}, upon which the proof of the lower bound is built, requires that $\what \in \mathcal{W}_{b}$. We formulate the bound in this weaker form.
}. First, observe that the term $m^{2}$ is matched by a zero predictor that corresponds to $\tilde{\omega} = 0$.
In Proposition~\ref{prop:lowerbound} presented in the previous section, we demonstrate a distribution with $m = 0$ (i.e., $Y = 0$ almost surely) under which a provably non-zero lower bound holds for some constrained least squares estimator. However, this is not the primary reason for the suboptimality that we establish in Theorem~\ref{thm:erm-ridge-lower-bound} with respect to the above lower bound and hence we ignore the term $m^2$ in what follows.

Second, for constrained least squares , the ``slow rate'' term $rbm/\sqrt{n} + r^{2}b^{2}/n$ is matched by optimizing the first term in Theorem~\ref{thm:main_thm} with the choice $\lambda = \frac{\sqrt{R(\omega^{*}_{b})n}}{rb}$. To see that, note that $\frac{1}{n}\widehat{\Sigma}_{\lambda r^{2}}^{-1} \preceq (\lambda r^{2})^{-1} I_{d}$ and hence
\begin{align*}
  \inf_{\lambda > 0}\left\{
    \frac{\E \xi^{2} X^{\mathsf{T}}\widehat{\Sigma}_{\lambda r^{2}}^{-1}X}{n} + \frac{\lambda r^{2}b^{2}}{n}
  \right\}
  &\leq 
  \inf_{\lambda > 0}\left\{
    \frac{\E \xi^{2} X^{\mathsf{T}}X}{\lambda r^{2}} + \frac{\lambda r^{2}b^{2}}{n}
  \right\}
  \\
  &\leq
  \inf_{\lambda > 0}\left\{
    \frac{\E \xi^{2}}{\lambda} + \frac{\lambda r^{2}b^{2}}{n}
  \right\}
  \leq
  2\sqrt{\frac{R(\omega^{*}_{b})r^{2}b^{2}}{n}},
\end{align*}
where the last line follows by noting that $\E \xi^{2} = R(\omega^{*}_{b})$ and plugging in the choice of $\lambda$ defined above. Finally, since $0 \in \mathcal{W}_{b}$, we have 
\[
R(\omega^{*}_{b}) \leq R(0) \leq m^{2},
\]
and the result follows.

Since the above discussion establishes that constrained least squares does not match the $m^{2}$ term in \eqref{eq:shamir-full} but matches the slow rate term, in what follows we focus on the fast rate term $dm^{2}/n + r^{2}b^{2}/n$, that is, the lower bound stated in \eqref{eq:shamir-simplified}. We demonstrate that the constrained least squares and the ridge regression estimators match this lower bound up to logarithmic factors under several assumptions widely considered in the literature. The key observation is that
$\E \xi^{2} = R(\omega^{*}_{b}) \leq m^{2}$ and
for any $\lambda > 0$, 
$\E X^{\mathsf{T}}\widehat{\Sigma}_{\lambda r^2}^{-1}X = \E \Tr(\widehat{\Sigma}_{\lambda r^2}^{-1}\widehat{\Sigma}_{0}) \leq d$. In particular, any independence-like assumption that allows to ``split'' the noise and the leverage scores in the localized multiplier term, perhaps at the price of extra logarithmic factors, establishes a form of optimality of the
constrained least squares and the ridge estimators with respect to the lower bound~\eqref{eq:shamir-simplified}.

We begin with the simplest example, which matches the lower bound \eqref{eq:shamir-simplified} under the assumption that
the noise random variables $\xi_{i}$ are independent of the covariates $X_{1}, \cdots, X_{n}$. Note that such an assumption is weaker than assuming that the model is well-specified, since we do not assume that $\xi_{i}$ are zero mean. To simplify the notation we write $\lambda$ instead of $\lambda r^2$ in what follows.
\begin{Example}
    \label{ex:independent-noise}
    Assume that the noise variable $\xi$ is independent of $X_{1}, \dots, X_{n}$. Then
    \begin{align}
        \frac{\E\xi^2X^\mathsf{T}\widehat{\Sigma}_{\lambda}^{-1}X}{n}
        &= \frac{\left(\E \xi^{2} \right)\big(\E X^\mathsf{T}\widehat{\Sigma}_{\lambda}^{-1}X\big)}{n}
        = \frac{R(\omega^*_{b}) \E \Tr(\widehat{\Sigma}_{\lambda}^{-1}\widehat{\Sigma}_{0})}{n} \\
        &\leq \frac{R(\omega^*_{b}) \Tr((\Sigma + \lambda I_d/n)^{-1}\Sigma)}{n},
    \end{align}
    where the last step follows by Jensen's inequality and the fact that
    that $A \mapsto \Tr((A + \lambda I_d/n)^{-1}A)$ is a concave map for $A \succeq 0$.
    Note that the quantity quantity $\Tr((\Sigma + \lambda I_d/n)^{-1}\Sigma)$, known as the \emph{effective dimension} (cf.\ \citep{hsu2014random}), is never larger than $d$.
\end{Example}

Our second example shows how to upper bound the multiplier term given an $L_{\infty}$ bound on the noise variable $\xi$. Among the prior work that proves upper bounds under such an assumption see \citep[Theorem 2.1]{audibert2011robust} and \citep[Equation (7)]{mourtada2019exact}. For a closely related assumption see \citep[Assumpton A6]{Bach2013NonstronglyconvexSS}, which is also imposed in order to decouple the noise variables from the statistical leverage scores in the spirit of moment equivalence assumptions.

\begin{Example}
  Let $\Sigma = \E XX^{\mathsf{T}}$. Then, for any $\lambda > 0$ we have
  \begin{align*}
    \frac{\E\xi^2X^\mathsf{T}\widehat{\Sigma}_{\lambda}^{-1}X}{n}
    \leq \frac{\norm{\xi}_{L_{\infty}}^{2} \E X^\mathsf{T}\widehat{\Sigma}_{\lambda}^{-1}X}{n}
    \leq
    \frac{\norm{\xi}_{L_{\infty}}^{2} \Tr((\Sigma + \lambda I_d/n)^{-1}\Sigma)}{n},
  \end{align*}
  where the last step follows by Jensen's inequality (cf.\ Example~\ref{ex:independent-noise}).
\end{Example}

We now turn to an example that requires a less restrictive control on the noise variables $\xi$.
\begin{Example}
  Let $\xi_{1}, \dots, \xi_{n}$ denote independent copies of $\xi$. Then, for any $\lambda > 0$ we have
  \begin{equation}
    \label{eq:max-l2-inequality}
    \frac{\E\xi^2X^\mathsf{T}\widehat{\Sigma}_{\lambda}^{-1}X}{n}
    \leq \frac{d\|\max_{i} \xi_{i}\|_{L_{2}}^{2}}{n}.
  \end{equation}
  The above inequality follows by noting that
  \begin{align*}
    \frac{\E\xi^2X^\mathsf{T}\widehat{\Sigma}_{\lambda}^{-1}X}{n}
    &\leq \frac{\E \left(\max_{i} \xi_{i}^{2}\right) \cdot \sum_{i=1}^{n}X_{i}^\mathsf{T}\widehat{\Sigma}_{\lambda}^{-1}X_{i}}{n^{2}}
    =
    \frac{\E \left(\max_{i} \xi_{i}^{2}\right) \cdot \Tr(\widehat{\Sigma}_{\lambda}^{-1}(n\widehat{\Sigma}_{0}))}{n^{2}}
  \end{align*}
  and using the fact that $\Tr(\widehat{\Sigma}_{\lambda}^{-1}(n\widehat{\Sigma}_{0})) \leq nd$.
\end{Example}

A sub-Gaussian norm of a random variable $Z$ is defined as (see e.g., Definition 2.5.6 in \citep{Vershynin2016HDP})
\[
    \|Z\|_{\psi_2} = \inf\left\{c > 0: \E\exp(Z^2/c^2) \le 2\right\}.
\]
Below, we show how an assumption that the noise $\xi$ is well-behaved yields to a simplification of the upper bound stated in the above example. We emphasize that the above assumption does not impose any restrictions on the covariates, other than boundedness assumption used throughout this paper.

\begin{Example}
    \label{ex:subgaussian}
    Suppose that the noise random variable $\xi$ satisfies the sub-Gaussian assumption of the form
    $\|\xi\|_{\psi_2} \lesssim \|\xi\|_{L_2}$. Then, by standard sub-Gaussian maximum inequalities (e.g., \citep{ledoux1991probability}), we have
    $\|\max_i \xi_i\|_{L_2}^2 \lesssim \|\xi\|_{L_2}^2\log n = R(\omega^*)\log n$
    and hence the upper bound \eqref{eq:max-l2-inequality} simplifies to
    \[
     \frac{\E\xi^2X^\mathsf{T}\widehat{\Sigma}_{\lambda}^{-1}X}{n}
        \lesssim \frac{dR(\omega^{*}_{b}) \log n}{n}.
    \]
\end{Example}

Our final example weakens the above assumption on the distribution of the noise $\xi$ but requires an $L_{4}$--$L_{2}$ moment equivalence for the marginals $\ip{\omega}{X}$. For related work proving excess risk bounds under similar moment equivalence assumptions see \citep[Assumption 2.1]{lugosi2016risk}, \citep[Theorem 1.2]{oliveira2016lower} and \citep[Assumptions 2 and 3]{mourtada2019exact}. The proof of the below proposition is based on controlling the lower tail of random quadratic forms using a result in \citep{oliveira2016lower}. See Appendix~\ref{sec:proof-of-l4l2} for details.
\begin{Proposition}
    \label{prop:l4l2}
    Suppose that $\Sigma = \E XX^{\mathsf{T}}$ is of full rank and assume that the following holds:
    \[
        \|\xi\|_{L_4} \lesssim \|\xi\|_{L_2}
        \quad\text{and for all }\omega\in\R^{d}\text{ we have}\quad
         \E\inr{\omega, X}^4 \lesssim \left(\E\inr{\omega, X}^2\right)^2.
    \]
    Then, for any $\lambda > 0$ and any $n \gtrsim d$, it holds that
    \begin{equation}
        \label{eq:l4-l2-example}
        \frac{\E\xi^2X^\mathsf{T}\widehat{\Sigma}_{\lambda}^{-1}X}{n}
        \lesssim \frac{dm^2}{n} + \frac{r^2b^2}{n}.
    \end{equation}
\end{Proposition}

As demonstrated in this section, the upper bound of order $dm^{2}/n + r^2b^2/n$ is achievable by the least squares and the ridge regression estimators under various distributional assumptions frequently considered in the literature. It is important to emphasize, however, that in Example \ref{ex:subgaussian} and Proposition \ref{prop:l4l2} we used the most favorable versions of moments equivalence assumptions, where the corresponding norms are linked via absolute constants. As we mentioned, in some cases these constants may depend on the dimension of the problem leading to suboptimal results. Indeed, once such assumptions are violated, we prove in Theorem~\ref{thm:erm-ridge-lower-bound} presented in the next section, that 
the optimistic rate $dm^{2}/n + r^2b^2/n$ is not always achievable. In contrast, the above rate can be achieved and surpassed by improper estimators (see Section~\ref{sec:exponential-improvement}).

\section{Main Results}
\label{sec:lower-bound}

In this section, we present our main result: a construction of a bounded distribution under which constrained least squares exceeds the lower bound \eqref{eq:shamir-simplified} by a factor proportional to $\sqrt{d}$. The regime considered in our lower bound is essential to establishing a separation between the performance of constrained least squares and non-linear estimators in our bounded setup. As discussed in Section~\ref{sec:corollaries}, various distributional assumptions ensure that the gap between the performance of constrained least squares and that of the lower bound \eqref{eq:shamir-simplified} is at most logarithmic.
As a result, we need to construct a bounded distribution that violates all the assumptions considered in Section~\ref{sec:corollaries}. 

If the noise variables and the leverage scores satisfy the assumptions that allow us to split them apart in the multiplier term, then we obtain an upper bound that matches the lower bound \eqref{eq:shamir-simplified}. Since Bernoulli random variables with a small parameter $p$ satisfy $L_4$--$L_2$ moment equivalence with an ill-behaved constant $1/p$, we aim to construct a distribution such that the noise random variables and the leverage scores both approximately follow Bernoulli distributions with a small parameter that depends on the dimension $d$. Besides, the noise variables and the leverage scores need to be \emph{highly correlated}; otherwise, the multiplier term would be too small. We remark that our construction could be considered somewhat extreme only with respect to the constants appearing in the moment equivalence assumptions discussed in Section~\ref{sec:corollaries}. At the same time, our distribution is bounded with the favorable choice of constants, the Bernstein class assumption is satisfied and therefore, by the upper bound \eqref{eq:excess-risk-suboptimal-b-plus-m}, the constrained least squares estimator satisfies non-trivial fast rate excess risk guarantee, making the construction of our main lower bound more challenging.
  
Let us now present our construction. For simplicity, we assume that $\sqrt{d}$ is an integer in what follows.
Let $\mathbf{1}$ denote an all-ones vector. For a support set $S \subseteq \{1, \dots, d\}$, let
$\mathbf{1}_{S}$ denote a vector such that $(\mathbf{1}_{S})_{i} = 1$ if $i \in S$ and $0$ otherwise.
Let $\mathcal{S}_{\sqrt{d}} = \{ S \subseteq \{1, \dots, d\} : \abs{S} = \sqrt{d} \}$.
We consider the following distribution:
\begin{equation}
\label{eq:bad-distribution}
  (X, Y)  = \begin{cases}
  (d^{-1}\mathbf{1}, 1) & \mbox{ with probability }1 - d^{-1/2}, \\
  (d^{-1/4}\mathbf{1}_{S}, 0) & \mbox{ with probability }d^{-1/2},\mbox{ where }
      S \sim \text{Uniform}\left(\mathcal{S}_{\sqrt{d}}\right).
  \end{cases}
\end{equation}
A simple calculation shows that $\omega^{*}_{\infty} \approx \frac{1}{2} \mathbf{1}$ and hence for $b \gtrsim \sqrt{d}$ we have
$\omega^{*}_{b} = \omega^{*}_{\infty}$. In particular, $\xi_{i}^{2}$ is smaller than $1$ for the ``high probability'' points
$(X_{i}, Y_{i}) = (d^{-1}\mathbf{1}, 1)$, while $\xi_{i}^{2} \approx \sqrt{d}$ for the ``low probability'' points 
$(X_{i}, Y_{i}) =  (d^{-1/4}\mathbf{1}_{S}, 0)$. This establishes that $\xi_{i}^{2}$ behaves as Bernoulli random variables. Similarly, since all the ``high probability'' points are exactly the same, they essentially have zero leverage.
On the other hand, the ``low probability'' points all have high leverage, thus the leverage scores also approximately follow the Bernoulli distribution. Finally, since $\xi_{i}^{2}$ is large exactly for the high leverage points, the squared noise random variables are correlated with the leverage scores. Intuitively, the multiplier term (i.e., the first term in Theorem~\ref{thm:main_thm}~and~\ref{thm:main_thm_second}) scales as $d^{3/2}/n$ under the distribution \eqref{eq:bad-distribution}, 
while the lower bound \eqref{eq:shamir-simplified} scales only as $d/n$ provided that $b \sim \sqrt{d}$.

The main result of our paper is presented below. The proof is deferred to Section~\ref{sec:erm-lower-bound-proof}.
\begin{Theorem}
\label{thm:erm-ridge-lower-bound}
    Suppose that the distribution $P$ of $(X, Y)$ is given by 
    \eqref{eq:bad-distribution}. Then, for any constrained least squares estimator $\what_{b}^{\operatorname{ERM}}$ defined by \eqref{eq:erm}, the following lower bound holds, provided that $d$ is large enough, $b \sim \sqrt{d}$ and 
    $n \gtrsim d^{3}\log d$:
    \[
    \E R(\what_{b}^{\operatorname{ERM}}) - R(\omega^*_b)
      \gtrsim \frac{d^{3/2}}{n}.
    \]
    For the distribution \eqref{eq:bad-distribution} we have $r = 1$ and $m=1$; hence, the lower bound
    \eqref{eq:shamir-simplified} scales only as $d/n$.
\end{Theorem}

We now comment on the above result. Recall that the aim of the construction \eqref{eq:bad-distribution} is to
maximize the multiplier term under boundedness constraints on the underlying distribution. In view of the lower bound \eqref{eq:shamir-simplified}, the parameters $m, r, b$ are chosen the most relevant way in the sense explained below. First, because of the homogeneity, we may always set $m = 1$. Second, the choice $b \sim \sqrt{d}$ is natural for $d$-dimensional vectors, particularly, for the underlying parameter $w^{*}_{b}$. Finally, the scaling $dm^{2} \sim r^{2}d^{2}$ equalizes the two terms in the lower bound \eqref{eq:shamir-simplified} and according to the results in Section \ref{sec:exponential-improvement} leaves open the possibility that in such regimes improper estimators offer no statistical improvements. Indeed, the best upper bound for non-linear estimators scales as $dm^{2}/n$. It also follows from the proof of Theorem~\ref{thm:erm-ridge-lower-bound} that under the distribution \eqref{eq:bad-distribution}, the constrained least squares estimator $\what_{b}^{\operatorname{ERM}}$ coincides with the global least squares solution, so that any larger value of $b$ can be chosen in Theorem \ref{thm:erm-ridge-lower-bound} without changing the statement. More importantly, our construction can be extended to a family of bounded distributions with the lower bound scaling as $d^{1 + \alpha}/n$ for any $\alpha \in [0, 1/2]$, while improper estimators (as shown in Section \ref{sec:exponential-improvement}) can still achieve the optimal $d/n$ rate. In all these cases, there is still a gap between the lower bound \eqref{eq:shamir-simplified} and the performance of constrained least squares. To simplify the presentation, we focus only on one particular distribution \eqref{eq:bad-distribution}, which maximizes the exhibited performance gap.

\subsection{Improvements via Non-Linear Predictors}
\label{sec:exponential-improvement}

In this section, we observe that non-linear predictors can surpass the lower bound \eqref{eq:shamir-simplified} that holds for linear predictors in $\mathcal{W}_{b}$. In particular, via known results in the literature, we first demonstrate that the Vovk-Azoury-Warmuth (VAW) forecaster yields an exponential improvement on the boundedness constants $b$ and $r$ compared to the lower bound \eqref{eq:shamir-simplified}. We remark that in our setup, the VAW forecaster was previously used to \emph{match} the lower bound \eqref{eq:shamir-simplified} in \citep{shamir2015sample}; the difference in the result below is that we tune the regularization parameter and use the resulting upper bound to demonstrate a statistical separation between proper and improper algorithms, rather than matching a lower bound that holds for proper algorithms. Finally, we discuss a less known modification of the VAW forecaster due to Forster and Warmuth \citep{forster2002relative} that can completely remove the dependence on the boundedness constants $b$ and $r$ and removes any assumptions on the distribution of the covariates.

The VAW forecaster is defined as follows. Given a (random) sample $S_{n}$, a regularization parameter $\lambda > 0$, and any point $X \in \mathbb{R}^{d}$ we first compute
\begin{equation}
\label{eq:vaw}
    \widehat{\omega}_{\lambda, n}(X) =
    \argmin\limits_{\omega \in \mathbb{R}^d}\sum\limits_{i = 1}^n(Y_i - \langle \omega, X_i\rangle)^2
    + \lambda\|\omega\|^2 + \langle \omega, X \rangle^2
\end{equation}
and then output a prediction 
\[
  \widehat{f}^{\operatorname{VAW}}_{\lambda}(X) =
  \langle\widehat{\omega}_{\lambda, n}(X), X \rangle.
\]
Thus, in order to make a prediction, a new linear predictor $\langle \what_{\lambda, n}(X), \cdot \rangle$ is computed for every point $X$ and in particular, the VAW forecaster is non-linear. For background on the VAW forecaster and regret bounds we refer to \citep*{vovk1998competitive, cesa2006prediction, orabona2019modern}. Below, for any predictor $f(\cdot)$ we denote $R(f(X)) = \E(Y - f(X))^2$.
Our key observation is that the sequence of weights $\widehat{\omega}_{\lambda^*, j}(X)$ for $j = 1, \ldots, n$ and $\lambda^* = \frac{dm^2}{b}$ can be immediately translated into a non-linear estimator $\widetilde{f}^{\operatorname{VAW}}_{\lambda^*}(\cdot)$ satisfying in the notation of the lower bound \eqref{eq:shamir-simplified}:
\begin{equation}
        \label{eq:orabona-vaw-second}
        \E R(\widetilde{f}^{\operatorname{VAW}}_{\lambda^*}(X)) - R(\omega^*_b) \lesssim \frac{dm^{2}}{n}\log\left(1 + \frac{r^{2}b^{2}n}{d^2m^2}\right).
\end{equation}
The proof of this fact follows from the regret bound for the VAW forecaster (see the survey \citep[Theorem 7.25]{orabona2019modern}) and the standard online-to-batch conversion. The above bound yields an exponential improvement on the boundedness constants $b$ and $r$ compared to the lower bound \eqref{eq:shamir-simplified}.

More importantly, there exists a modification of the VAW forecaster due to \citep{forster2002relative} that can remove the logarithmic factor in \eqref{eq:orabona-vaw-second}.\footnote{
  We are thankful to Manfred Warmuth for pointing us to the modified VAW forecaster.
} Let us introduce the modified VAW forecaster. 
Given a sample $S_{n}$ and any $X \in \R^{d}$, let $h_{X} = X^{\mathsf{T}} (\sum_{i=1}^{n} X_{i}X_{i}^{\mathsf{T}} + XX^{\mathsf{T}})^{\dagger}X$ denote the leverage of the point $X$ with respect to
the covariates $X_{1}, \dots, X_{n}, X$, where the notation $A^{\dagger}$ denotes the Moore-Penrose inverse of a matrix $A$. The modified VAW predictor $\widehat{f}$ is then defined pointwise as follows:
\begin{equation}
    \label{eq:modified-vaw-pointwise}
    \widehat{f}(X) = (1 - h_{X})\widehat{f}^{\operatorname{VAW}}_{0}(X).
\end{equation}
Thus, the above function outputs the predictions of the VAW forecaster (with $\lambda = 0$, where the VAW predictions are computed by taking a Moore-Penrose inverse of the sample covariance matrix), albeit reweighted by the factor $(1 - h_{X})$. Intuitively, the above predictor avoids making large errors for high leverage points. The following theorem describes the main property of this estimator.

\begin{theorem*}[Theorem 6.2 in \citep{forster2002relative}]
Let $\widehat{f}(\cdot)$ denote the non-linear predictor defined in \eqref{eq:modified-vaw-pointwise}. Let $P$ be any distribution (with possibly unbounded covariates) satisfying $\|Y\|_{L_{\infty}} \le m$. Then, for any $d, n > 0$, the following holds:
\begin{equation}
  \label{eq:optimal_bound}
    \E R(\widehat{f}(X)) - \inf\limits_{\omega \in \mathbb{R}^d}R(\omega)
        \lesssim \frac{dm^{2}}{n}.
\end{equation}
\end{theorem*}
\par
For the reader's convenience, we reproduce the proof of the above result in Appendix~\ref{section:proof-of-optimal-bound} and compare some of the steps with the proof of Theorem~\ref{thm:main_thm_second}. Observe that the bound \eqref{eq:optimal_bound} is closely related to the upper bound \eqref{eq:well-specified-excess-risk} that holds for unconstrained least squares in the well-specified setup with Gaussian design: both bounds do not depend on the magnitude of the covariates, specific properties of the covariance structure, and the norm of the optimal linear predictor $\omega^* = \arginf_{\omega\in \mathbb{R}^d}R(\omega)$. The difference is that $R(\omega^*)$ is replaced by  $m^2$. It is reported in \citep{forster2002relative} that the authors could not prove a bound similar to \eqref{eq:optimal_bound} for the least squares estimator. Given the lower bound~\eqref{eq:shamir-simplified} and Proposition~\ref{prop:lowerbound}, it is not surprising. In our setting, the performance of least squares is affected by the boundedness constants $r, b$, which can be arbitrarily bad in the theorem above. 

Finally, note that once the pseudoinverse of the sample covariance matrix is computed, pointwise evaluation of the non-linear estimator \eqref{eq:non-linear-predictor-prediction} can be done in $O(d^{2})$ operations. In contrast, only $O(d)$ operations are needed to evaluate a linear function. It is unknown whether a more computationally efficient algorithm that matches the upper bounds \eqref{eq:orabona-vaw-second} and \eqref{eq:optimal_bound} exists and, more broadly, whether there exist inherent statistical-computational trade-offs needed to attain the optimal rate in the distribution-free setting. The search for computationally efficient improper algorithms in a related phenomenon observed for logistic regression (cf.\ Section~\ref{sec:related-work}) is currently an active line of research \citep{pmlr-v75-foster18a, mourtada2019improper, jezequel2020efficient}.

\section{Proofs}
\label{sec:proofs}

\subsection{Proof of Theorem \ref{thm:main_thm}}
\label{section:erm-proof}

This section is devoted to the proof of Theorem~\ref{thm:main_thm}. First, notice that by convexity of the quadratic loss and convexity of the class $\mathcal{W}_{b}$, the following inequality holds sometimes called the \emph{Bernstein condition} in the literature:
\begin{equation}
    \label{eq:bernstein_eq}
    R(\omega) - R(\omega^*_b) \geq \E(\langle \omega - \omega_{b}^{*}, X\rangle)^2 = \norm{\omega - \omega^{*}_{b}}_{L_{2}}^{2}
\end{equation}

Our analysis is split into three parts. First, we provide the excess risk bound in terms of the localized complexities corresponding to the quadratic and multiplier terms. Then, we prove sharp bounds for both of them. To simplify the notation, in what follows we write $\what, \omega^{*}$ instead of $\what_{b}^{\operatorname{ERM}}, \omega^{*}_{b}$.

\myparagraph{Localization}
For any $\omega, x \in \mathbb{R}^d, y \in \mathbb{R}$ define the excess loss functional
\[
{\cal L}_\omega(x,y) = (\inr{\omega,x}-y)^2-(\inr{\omega^*,x}-y)^2.
\]
Let us split the empirical excess risk to the quadratic and multiplier components as follows:
\begin{align*}
    P_n {\cal L}_\omega = & \frac{1}{n} \sum_{i=1}^n {\cal L}_\omega(X_i,Y_i)
    = \underbrace{\frac{1}{n} \sum_{i=1}^n \inr{\omega-\omega^*,X_i}^2}_{P_n {\cal Q}_{\omega-\omega^*}} + \underbrace{\frac{2}{n}\sum_{i=1}^n (\inr{\omega^*,X_i}-Y_i)\cdot \inr{\omega-\omega^*,X_i}}_{P_n {\cal M}_{\omega-\omega^*}},
\end{align*}
where $P_n {\cal Q}_{\omega}$ and $P_n {\cal M}_{\omega}$ denote empirical quadratic and multiplier processes respectively, both indexed by $\omega \in \R^{d}$. We denote their population counterparts by
$$
    \E {\cal Q}_{\omega-\omega^*} = \norm{\omega - \omega^{*}}^{2}_{L_{2}}
    \quad\text{and}\quad
    \E {\cal M}_{\omega-\omega^*} = \E2(\inr{\omega^*,X}-Y)\cdot \inr{\omega-\omega^*,X}.
$$
Since $\widehat{\omega}$ defined in \eqref{eq:erm} minimizes the empirical excess risk $P_{n}{\mathcal{L}_{\omega}}$ over $\omega \in \mathcal{W}_{b}$, we have $P_n {\cal L}_{\widehat{\omega}} \leq 0$. Thus, it suffices to show that, with high probability, if $\E {\cal L}_\omega \geq q^2$ for some $q > 0$, then $P_n {\cal L}_\omega >0$. This will imply by contradiction that with high probability $\E {\cal L}_{\widehat{\omega}} = R(\widehat{\omega}) - R(\omega^*) \le q^2$.

Recall that $\|\omega\|_{L_2}^2 = \E \inr{\omega, X}^2$. As a first step, let us show that if $\|\omega-\omega^*\|^2_{L_2}$ is larger than a maximum of suitably defined fixed points (see below), then $P_n {\cal L}_\omega > 0$, thus implying that $\|\what-\omega^*\|^2_{L_2}$ is small.

Fix some $\omega \in \mathcal{W}_{b}$ and let $s^2 = \|\omega - \omega^*\|^2_{L_2}$. We aim to investigate under what assumptions on $s$ it holds that $P_n {\cal L}_\omega > 0$. Using the Bernstein assumption \eqref{eq:bernstein_eq} we have
 \begin{align}
    \label{eq:excess-loss-decomposition}
    P_{n} {\cal L}_{\omega} &= P_n {\cal Q}_{\omega-\omega^*} + P_n {\cal M}_{\omega-\omega^*} 
    \\
    = & \left(P_n {\cal Q}_{\omega-\omega^*} - \E {\cal Q}_{\omega-\omega^*}\right) + \left(P_n {\cal M}_{\omega-\omega^*} - \E {\cal M}_{\omega-\omega^*}\right) + \E {\cal L}_\omega
    \\
    \geq & \left(P_n {\cal Q}_{\omega-\omega^*} - \E {\cal Q}_{\omega-\omega^*}\right) + \left(P_n {\cal M}_{\omega-\omega^*} - \E {\cal M}_{\omega-\omega^*}\right) + s^{2}.
\end{align}
Observe that if $s$ satisfies
\begin{align}
    \label{eq:fixedpoints}
    &\sup_{\omega \in \mathcal{W}_{b}, \|\omega - \omega^*\|_{L_2} \le s} \left|P_n {\cal Q}_{\omega-\omega^*} - \E {\cal Q}_{\omega-\omega^*}\right| \leq \frac{s^2}{10} \nonumber
    \\
    &\quad\text{and}\quad
    \sup_{\omega \in \mathcal{W}_{b}, \|\omega - \omega^*\|_{L_2} \le s} \left|P_n {\cal M}_{\omega-\omega^*} - \E {\cal M}_{\omega-\omega^*}\right| \leq \frac{s^2}{10},
\end{align}
then $P_{n} {\cal L}_{\omega} \geq -\frac{2s^{2}}{10} + s^{2} > 0$. The inequality $P_{n} {\cal L}_{\omega} > 0$ extends to all $\omega' \in \mathcal{W}_{b}$ such that $\|\omega^{\prime}  - \omega^*\|_{L_2} \ge s$ via a standard star-shapedness argument of the class $\mathcal{W}_{b}$. To see that, suppose that 
$\|\omega^{\prime}  - \omega^*\|_{L_2} = s' > s$. Then,
\begin{equation}
    \label{eq:star-shapedness}
    \omega' - \omega^{*}
    =  \frac{s'}{s}\left(
        \underbrace{\frac{s}{s'}\omega' + \frac{s' - s}{s'}\omega^{*}}_{\omega'_{s}} - \omega^{*}
        \right)
    = \frac{s'}{s}\left(\omega'_{s} - \omega^{*}\right).
\end{equation}
By convexity of $\mathcal{W}_{b}$, $\omega'_{s} \in \mathcal{W}_{b}$. Further, 
$\|\omega^{\prime}_{s}  - \omega^*\|_{L_2} = \frac{s}{s'}\|\omega^{\prime}  - \omega^*\|_{L_2} = s$.
Hence, it follows that
\begin{align*}
    P_{n} {\cal L}_{\omega'}
    &= P_n {\cal Q}_{\omega'-\omega^*} + P_n {\cal M}_{\omega'-\omega^*}
    \\
    &= \left(\frac{s'}{s}\right)^{2}P_n {\cal Q}_{\omega'_{s}-\omega^*} + \frac{s'}{s}P_n {\cal M}_{\omega'_{s}-\omega^*} \geq
    \frac{s'}{s}\left(P_n {\cal Q}_{\omega'_{s}-\omega^*} + P_n {\cal M}_{\omega'_{s}-\omega^*}\right)
    > 0.
\end{align*}

Therefore, we are interested in the smallest value of $s$ that satisfies the two conditions in \eqref{eq:fixedpoints}. This leads to the definition of the (random) fixed point corresponding to the quadratic term:
\begin{equation}
\label{eq:rq}
    s^*_{\cal Q} = \inf\left\{s > 0:\;\sup_{\omega \in \mathcal{W}_{b}, \|\omega - \omega^*\|_{L_2} \le s} \left|P_n {\cal Q}_{\omega-\omega^*} - \E {\cal Q}_{\omega-\omega^*}\right| \leq \frac{s^2}{10}\right\}
\end{equation}
 and the (random) fixed point corresponding to the multiplier term:
\begin{equation}
\label{eq:rm}
    s^*_{\cal M} = \inf\left\{s > 0:\;\sup_{\omega \in \mathcal{W}_{b}, \|\omega - \omega^*\|_{L_2} \le s} \left|P_n {\cal M}_{\omega-\omega^*} - \E {\cal M}_{\omega-\omega^*}\right| \leq \frac{s^2}{10}   \right\}.
\end{equation}

We conclude by proving the excess risk bound for the constrained least squares estimator $\widehat{\omega}$ in terms of the fixed points $s^*_{\cal M}$ and $s^*_{\cal Q}$ defined above.
Observe that if $\|\omega-\omega^*\|^2_{L_2} \ge \max\{s^*_{\cal M}, s^*_{\cal Q}\}^2$, then by the argument above, $P_n {\cal L}_\omega > 0$. Therefore, $\widehat{\omega}$ satisfies $\|\widehat{\omega}-\omega^*\|^2_{L_2} \le \max\{s^*_{\cal M}, s^*_{\cal Q}\}^2$. Assume for the sake of contradiction that  $\E {\cal L}_{\widehat{\omega}}  > 3\max\{s^*_{\cal M}, s^*_{\cal Q}\}^2$. Then, 
\begin{align*}
    3\max\{s^*_{\cal M}, s^*_{\cal Q}\}^2 < \E {\cal L}_{\widehat{\omega}}
    = \E {\cal Q}_{\what-\omega^*} + \E {\cal M}_{\what-\omega^*}
    \le \max\{s^*_{\cal M}, s^*_{\cal Q}\}^2 + \E {\cal M}_{\what-\omega^*}.
\end{align*}
The above inequality implies that $\E {\cal M}_{\what-\omega^*} > 2\max\{s^*_{\cal M}, s^*_{\cal Q}\}^2$. Therefore, using the fact that
$P_n {\cal L}_{\what} \ge P_n {\cal M}_{\what-\omega^*} = \E {\cal M}_{\what-\omega^*} +  P_n {\cal M}_{\what-\omega^*} - \E {\cal M}_{\what-\omega^*}$ we have
\begin{align}
    P_n {\cal L}_{\widehat{\omega}} &\ge  P_n {\cal M}_{\omega-\omega^*} - |P_n {\cal M}_{\widehat{\omega}-\omega^*} - \E {\cal M}_{\widehat{\omega}-\omega^*}|
    \\
    &\ge 2(\max\{s^*_{\cal M}, s^*_{\cal Q}\})^2 - |P_n {\cal M}_{\widehat{\omega}-\omega^*} - \E {\cal M}_{\widehat{\omega}-\omega^*}|.
    \label{eq:excess-risk-bound-localization-penultimate-step}
\end{align}
We now aim to upper bound $|P_n {\cal M}_{\widehat{\omega}-\omega^*} - \E {\cal M}_{\widehat{\omega}-\omega^*}|$ in order to conclude that $P_{n} {\cal L}_{\what} > 0$, which will yield the desired contradiction.
Recall that $\norm{\what - \omega^{*}}_{L_{2}}^{2} \leq \max\{s^*_{\cal M}, s^*_{\cal Q}\}^2$.
By the star-shapedness argument (cf.\ Equation~\eqref{eq:star-shapedness}), for all $\omega \in \mathcal{W}_{b}$ with
$s^{*}_{\mathcal{M}} < \norm{\omega - \omega^*}_{L_{2}} \leq \max\{s^*_{\cal M}, s^*_{\cal Q}\}$, 
there exists some $\omega' \in \mathcal{W}_{b}$ such that $\omega - \omega^{*} = \frac{\max\{s^*_{\cal M}, s^*_{\cal Q}\}}{s^*_{\cal M}}(\omega' - \omega^{*})$ and $\norm{\omega' - \omega^*}_{L_{2}} \leq s^*_{\cal M}$.
To simplify the notation, for any $r > 0$, denote
$
 \mathcal{B}(r) = \{ \omega \in \mathcal{W}_{b} : \|\omega - \omega^*\|_{L_2} \leq  r \}.
$
Then, we have
\begin{align}
    |P_n {\cal M}_{\widehat{\omega}-\omega^*} - \E {\cal M}_{\widehat{\omega}-\omega^*}| 
    &\leq \sup_{\omega \in \mathcal{B}(\max\{s^*_{\cal M}, s^*_{\cal Q}\})} \left|P_n {\cal M}_{\omega-\omega^*} - \E {\cal M}_{\omega-\omega^*}\right| \\
    &\leq
     \frac{\max\{s^*_{\cal M}, s^*_{\cal Q}\}}{s^*_{\cal M}} \cdot \left(
    \sup_{\omega \in \mathcal{B}(s^*_{\cal M})} \left|P_n {\cal M}_{\omega-\omega^*} - \E {\cal M}_{\omega-\omega^*}\right| \right) 
    \\
    &\leq
    \frac{(\max\{s^*_{\cal M}, s^*_{\cal Q}\})^2}{10}.
\end{align}
Combining the above inequality with \eqref{eq:excess-risk-bound-localization-penultimate-step} yields
$P_{n} {\cal L}_{\what} > 0$, which contradicts the assumption that $\what$ is an empirical risk minimizer
over $\mathcal{W}_{b}$. Therefore, we have
\begin{equation}
\label{eq:excessrisk}
    \E R(\widehat{\omega}) - R(\omega^*)
    \le 3\E (\max\{s^*_{\cal M}, s^*_{\cal Q}\})^2
    \le 3\E (s^*_{\cal M})^2 + 3\E (s^*_{\cal Q})^2
\end{equation}
and we turn to the upper bounds on $\E (s^*_{\cal Q})^2$ and $\E (s^*_{\cal M})^2$ in the sequel.

\myparagraph{Quadratic term}
In this part of the analysis we obtain an upper bound on $\E(s^*_{\cal Q})^2$. Denote the second moment matrix by $\Sigma = \E XX^{\mathsf{T}}$ and assume without loss of generality that $\lambda_{1}^2 \ge \lambda_{2}^2 \ge \ldots \ge \lambda_{d}^2 > 0$. Indeed, if some of the eigenvalues are equal to zero then the distribution of $X$ is supported on a subspace of $\mathbb{R}^d$. Then we may restrict our analysis to this subspace only.

We may write $X = \Sigma^{\frac{1}{2}}Z$, where $Z$ is an \emph{isotropic} vector ($\E ZZ^{\mathsf{T}} = I_{d}$) and the eigenvalues of $\Sigma^{\frac{1}{2}}$ satisfy $\lambda_{1} \ge  \lambda_{2} \ge \ldots \ge \lambda_{d} > 0$. Observe that
\[
    \|\omega - \omega^*\|^2_{L_2} \le s^2 \quad \text{is equivalent to}\quad (\omega - \omega^*)^{\mathsf{T}}\Sigma(\omega - \omega^*) \le s^2.
\]
Denoting $v = \Sigma^{\frac{1}{2}}(\omega - \omega^*)$ and  $\mathcal{V} = \{\Sigma^{\frac{1}{2}}(\omega - \omega^*): \omega \in \mathcal{W}_{b}\}$ we may write
\begin{align}
    &\sup_{\omega \in \mathcal{W}_{b}, \|\omega - \omega^*\|_{L_2} \le s} \left|\frac{1}{n}\sum_{i=1}^n\inr{X_i, \omega - \omega^*}^2 - \E \inr{X,\omega - \omega^*}^2 \right| \nonumber
    \\
    &\quad\quad= \sup_{v \in \mathcal{V}, \|v\| \le s}  \left|\frac{1}{n}\sum_{i=1}^n \inr{Z_i, v}^2 - \E \inr{Z, v}^2 \right|. \label{eq:quadraticeq}
\end{align}
In what follows, our idea is to replace the supremum over the set $\{v: v \in \mathcal{V}, \|v\| \le s\}$ in \eqref{eq:quadraticeq} by the supremum over the unit ball by considering a special vector $W$ defined below which replaces $X$ (and $Z$). This will put us in position to apply the concentration result of Oliveira \citep{oliveira2010sums}.

Since the matrix $\Sigma$ is real and symmetric, we may write $\Sigma = U^T\text{Diag}(\lambda^2_1, \ldots, \lambda^2_d)U$, where $U$ is an orthogonal $d \times d$ matrix, and therefore, $\Sigma^{\frac{1}{2}} = \text{Diag}(\lambda_1, \ldots, \lambda_d)U$.
Denote by $B(s)$ the closed Euclidean ball in $\mathbb{R}^d$ of radius $s$ centred at zero. Since $\|\omega - \omega^*\| \le 2b$ we have
\[
    \{v: v \in \mathcal{V}, \|v\| \le s\} \subseteq B(s) \cap \Sigma^{\frac{1}{2}}B(2b) =  B(s) \cap \text{Diag}(\lambda_{1}, \ldots, \lambda_d)B(2b),
\]
where the last inequality holds since the orthogonal matrix does not change the Euclidean ball $B(2b)$. Let $e_1, \ldots, e_d$ denote the standard basis in $\mathbb{R}^d$. It is easy to verify that any point $(x_1, \ldots, x_d)$ that belongs to the intersection of the ball and the ellipsoid $B(s) \cap \text{Diag}(\lambda_{1}, \ldots, \lambda_d)B(2b)$ satisfies for any $1 \le k \le d$,
\begin{equation}
    \label{eq:ellipsoid}
    \sum_{i=1}^k x_ie_i + \sum_{i=k+1}^d \left(\frac{s}{2b \lambda_i} \right) x_i e_i \in B(2s).
\end{equation}
Indeed, we have $\sum\limits_{i = 1}^dx_i^2 \le s^2$ and $\sum\limits_{i = 1}^d\frac{s^2x_i^2}{4b^2\lambda^2_i} \le s^2$ which leads to $\sum\limits_{i = 1}^kx_i^2 \le s^2$ and $\sum\limits_{i = k + 1}^d\frac{s^2x_i^2}{4b^2\lambda^2_i} \le s^2$ implying \eqref{eq:ellipsoid}. In what follows, let $k$ be the largest integer that satisfies $s \leq 2b \lambda_k$. Finally, the set consisting of all $(x_1, \ldots, x_d)$ satisfying \eqref{eq:ellipsoid} contains $\{v: v \in \mathcal{V}, \|v\| \le s\}$ as a subset.
Denote $Z = (z_1, \ldots, z_d)$. For the same value of $k$ define the random vector $W = (w_1, \ldots, w_d)$,
\[
    W=\left(z_1, \ldots, z_k, \frac{2b \lambda_{k + 1}}{s}z_{k + 1}, \ldots, \frac{2b \lambda_n}{s}z_n\right).
\]
We may rewrite 
\[
    \inr{Z, v}=\sum_{i = 1}^k v_i z_i + \sum_{i=k+1}^d \left(\frac{s}{2b \lambda_i} \right) v_i\left(\frac{2b \lambda_i}{s}\right)z_i = \sum_{i=1}^k v_i w_i + \sum_{i=k+1}^d \left(\frac{s}{2b \lambda_i} \right) v_i w_i.
\]
These computations imply that
\begin{equation}
    \label{eq:almostmatrixinequality}
    \sup_{v \in \mathcal{V}, \|v\| \le s}  \left|\frac{1}{n}\sum_{i=1}^n \inr{Z_i, v}^2 - \E \inr{Z, v}^2 \right| \le \sup_{v \in B(2s)} \left|\frac{1}{n}\sum_{i=1}^n \inr{W_i,v}^2 - \E \inr{W,v}^2 \right|
\end{equation}
Finally, we provide two properties of the defined random vector $W$:
\begin{itemize}
    \item Note that $\sum\limits_{i = 1}^d \lambda_i^2z_i^2 = \|\Sigma^{1/2}Z\|^{2} = \|X\|^{2} \leq 1$ almost surely. We have $\|W\|^2 = \sum_{i \leq k} \frac{1}{\lambda_i^2} \cdot (\lambda_i^2 z_i^2) + \frac{4b^2}{s^2} \sum_{i=k+1}^d \lambda_i^2 z_i^2 \leq \max\left\{\frac{1}{\lambda_k^2}, \frac{4b^2}{s^2}\right\},$
    and recalling that $s \leq 2\lambda_k b$, it implies that almost surely
    \begin{equation}
        \label{eq:wbound}
        \|W\| \leq \frac{2b}{s}.
    \end{equation}
    
    \item For every $v \in \R^d$, since $2\lambda_i b \leq s$ for $i \geq k+1$,
    \begin{equation} \label{eq:prop-of-W-1}
        \E \inr{W,v}^2 = \sum_{i \leq k} v_i^2 + \sum_{i=k+1}^d v_i^2 \left(\frac{2b \lambda_i}{s}\right)^2 \leq \|v\|_2^2.
    \end{equation}
\end{itemize}
Combining these two properties with \eqref{eq:almostmatrixinequality} we apply a version of Rudelson's inequality for rank one operators \citep[Lemma 1]{oliveira2010sums} which implies that with probability at least $1 - \delta$
\begin{align*}
    \sup_{v \in B(2s)} \left|\frac{1}{n}\sum_{i=1}^n \inr{W_i,v}^2 - \E \inr{W,v}^2 \right| &= 4s^2\sup_{v \in B(1)} \left|\frac{1}{n}\sum_{i=1}^n \inr{W_i,v}^2 - \E \inr{W,v}^2 \right|
    \\
    &\le 32bs\sqrt{\frac{2\log(\min\{n, d\}) + 2\log 2 + \log \frac{1}{\delta}}{n}},
\end{align*}
provided that $\frac{8b}{s}\sqrt{\frac{2\log(\min\{n, d\} + 2\log 2 + \log \frac{1}{\delta})}{n}} \le 2$. Recalling the definition \eqref{eq:rq} of $s^*_{\cal Q}$ and solving the fixed point inequality $32bs\sqrt{\frac{2\log(\min\{n, d\} + 2\log 2 + \log \frac{1}{\delta})}{n}} \le \frac{s^2}{10}$, we may choose a large enough numerical constant $c_1$ such that, with probability at least $1 - \delta$,
\begin{equation}
    \label{eq:quadratichp}
    (s^*_{\cal Q})^2 \le  c_1\frac{b^2(\log(\min\{n, d\}) + \log \frac{1}{\delta})}{n}.
\end{equation}
Integrating the last inequality we have for $u^{\prime} = c_2\frac{b^2\log(\min\{n, d\})}{n}$, where $c_2$ is some numerical constant
\begin{align*}
    \E(s^*_{\cal Q})^2 &= \int\limits_{0}^{\infty}\Pr\left((s^*_{\cal Q})^2 > u\right)du \le u^{\prime} + \int\limits_{u^{\prime}}^{\infty}\exp\left(-\frac{nu}{b^2c_1} + \log(\min\{n, d\})\right)du 
    \\
    &\lesssim \frac{b^2\log(\min\{n, d\})}{n}.
\end{align*}
And for $r > 0$ we have due to homogeneity 
\begin{equation}
    \label{eq:quadraticexp}
    \E(s^*_{\cal Q})^2 \lesssim \frac{r^2b^2\log(\min\{n, d\})}{n}.
\end{equation}
\myparagraph{Multiplier term}
In this part of the proof we work with general $r > 0$ and we aim to upper bound $\E (s_{\mathcal{M}}^{*})^2$. Recall that $\xi = Y - \inr{\omega^*,X}$. Fix any $\lambda > 0$ and consider the event $E$ that $s^*_{\cal M}/2 > s^*_{\cal Q}$. 
Denote $\mathcal{W}^{\prime} = \{\omega: \omega \in \mathcal{W}_{b},\; \|\omega - \omega^*\|_{L_2} \le s^*_{\cal M}/2\}$
Plugging $s = s^*_{\cal M}/2$ into \eqref{eq:rm} we have on $E$,
\begin{align*}
    (s^*_{\cal M})^2 &\le \sup_{\omega \in \mathcal{W}^{\prime}} 80\left|\frac{1}{n}\sum_{i=1}^n \xi_i\inr{X_i, \omega^* - \omega} - \E\xi\inr{X, \omega^* - \omega}\right|
    \\
    &\le \sup_{\omega \in\mathcal{W}^{\prime}}\Biggl(80\left|\frac{1}{n}\sum_{i=1}^n \xi_i\inr{X_i, \omega^* - \omega} - \E\xi\inr{X, \omega^* - \omega}\right|
    \\
    &\quad\quad\quad+\|\omega - \omega^*\|_{L_2}^2 - \frac{1}{n}\sum_{i=1}^n\inr{X_i, \omega - \omega^*}^2 - \frac{2\lambda r^2\|\omega - \omega^*\|^2}{n}\Biggr) 
    \\
    &\quad\quad\quad\quad+\frac{(s^*_{\cal M})^2}{40} + \frac{8\lambda r^2b^2}{n}
    \\
    &\le \sup_{\omega \in \mathcal{W}^{\prime}}\Biggl(80\left|\frac{1}{n}\sum_{i=1}^n \xi_i\inr{X_i, \omega^* - \omega} - \E\xi\inr{X, \omega^* - \omega}\right|
    \\
    &\quad\quad\quad-\|\omega - \omega^*\|_{L_2}^2 - \frac{1}{n}\sum_{i=1}^n\inr{X_i, \omega - \omega^*}^2 - \frac{2\lambda r^2\|\omega - \omega^*\|^2}{n}\Biggr) 
    \\
    &\quad\quad\quad\quad+ (s^*_{\cal M})^2\left(\frac{1}{2} + \frac{1}{40}\right) + \frac{8\lambda r^2b^2}{n}.
\end{align*}
In the first step above we used the definition \eqref{eq:rq} of $s^*_{\cal Q}$ together with the star-shapedness argument (cf.\ the localization part of the proof above), the inequality $s^*_{\cal M}/2 > s^*_{\cal Q}$, and the fact that $\|\omega - \omega^*\|^2 \le 4b^2$ for $\omega \in \mathcal{W}_{b}$. In the second inequality we have used the definition of $\mathcal{W}^{\prime}$.

Let $\varepsilon_1, \ldots, \varepsilon_n$ be independent random signs and let $\E^{\prime}$ denote the expectation with respect to an independent copy of the sample $S_{n}$. We now show how to control the multiplier term on the event $E$, that is, the term $\E (s^*_{\cal M})^2\ind_{E}$. By Jensen's inequality, the symmetrization argument and the symmetry of $\mathbb{R}^d$ used to remove the absolute value, we have the following: 
\begin{align*}
    &\E\sup_{\omega \in \mathcal{W}^{\prime}}\Biggl(\left|\frac{80}{n}\sum_{i=1}^n \xi_i\inr{X_i, \omega^* - \omega} - \E^{\prime}\xi\inr{X, \omega^* - \omega}\right| 
    \\
    &\quad\quad\quad-\frac{1}{n}\sum_{i=1}^n\inr{X_i, \omega - \omega^*}^2 - \E^{\prime}\xi\inr{X, \omega^* - \omega} -  \frac{2\lambda r^2\|\omega - \omega^*\|^2}{n}\Biggr)
    \\
    &\le \E\sup_{v \in \mathbb{R}^d}\left(\left|\frac{80}{n}\sum_{i=1}^n \xi_i\inr{X_i, v} - \E^{\prime}\xi\inr{X, v}\right| - \frac{1}{n}\sum_{i=1}^n\inr{X_i, v}^2 - \E^{\prime}\xi\inr{X, v} -  \frac{2\lambda r^2\|v\|^2}{n}\right)
    \\
    &\le\E\E^{\prime}\sup_{v \in \mathbb{R}^d}\Biggl(\frac{80}{n}\left|\sum_{i=1}^n\varepsilon_i(\xi_i\inr{X_i, v} - \xi_i^{\prime}\inr{X_i^{\prime}, v})\right| 
    \\
    &\quad\quad\quad- \frac{1}{n}\sum_{i=1}^n\inr{X_i, v}^2 -\frac{1}{n}\sum_{i=1}^n\inr{X_i^{\prime}, v}^2 -  \frac{2\lambda r^2\|v\|^2}{n}\Biggr)
    \\
    &\le2\E\sup_{v \in \mathbb{R}^d}\left(\frac{80}{n}\sum_{i=1}^n\varepsilon_i\xi_i\inr{X_i, v}  - \frac{1}{n}\sum_{i=1}^n\inr{X_i, v}^2 -  \frac{\lambda r^2\|v\|^2}{n}\right)
    \\
    &=\frac{3200}{n}\E\left(\sum\limits_{i, j}^n\varepsilon_i\varepsilon_j\xi_i\xi_jX_i^{\mathsf{T}}\left(\lambda r^2 I_d + \sum\limits_{k = 1}^nX_kX_k^{\mathsf{T}}\right)^{-1}X_j\right)
    \\
    &= \frac{3200}{n}\E\sum\limits_{i = 1}^n\xi_i^2X_i^{\mathsf{T}}\left(\lambda r^2 I_d + \sum\limits_{k = 1}^nX_kX_k^{\mathsf{T}}\right)^{-1}X_i
    \\
    &= \frac{3200}{n}\E\xi^2X^{\mathsf{T}}\widehat{\Sigma}_{\lambda r^2}^{-1}X,
\end{align*}
where in the last lines we used the exact value of $v$ maximizing the expression as well as the exchangeability of $X_i$ and $X_j$.
Finally, we have
\begin{align*}
    \E(s^*_{\cal M})^2 &= \E (s^*_{\cal M})^2\ind_{\{s^*_{\cal M} \le 2s^*_{\cal Q}\}} + \E (s^*_{\cal M})^2\ind_{\{s^*_{\cal M} > 2s^*_{\cal Q}\}}
    \\
    &\le 4\E (s^*_{\cal Q})^2 + \frac{3200}{n}\E\xi^2X^{\mathsf{T}}\widehat{\Sigma}_{\lambda r^2}^{-1}X + \E (s^*_{\cal M})^2\left(\frac{1}{2} + \frac{1}{40}\right) + \frac{8\lambda r^2 b^2}{n}.
\end{align*}
Combining the last inequality with \eqref{eq:excessrisk} and \eqref{eq:quadraticexp} we have
\begin{equation}
    \label{abc}
    \E R(\what_{b}^{\operatorname{ERM}}) - R(\omega^*) \lesssim \frac{\E\xi^2X^{\mathsf{T}}\widehat{\Sigma}_{\lambda r^2}^{-1}X}{n}+ \frac{(\lambda + \log(\min\{n, d\}))r^2b^2}{n}.
\end{equation}
This proves Theorem \ref{thm:main_thm}.
\qed

\subsection{Proof of Theorem \ref{thm:erm-ridge-lower-bound}}
\label{sec:erm-lower-bound-proof}
The proof of this result is split into several steps. In this section, we provide three technical lemmas and demonstrate how they imply the result of Theorem~\ref{thm:erm-ridge-lower-bound}. The first two lemmas are based on exact computations using the Sherman-Morrison formula. The proof of the third lemma, for which we sketch a simple heuristic argument before presenting the formal proof, is based on matrix concentration inequalities. We always assume that $d$ (and therefore $n$, since it satisfies $n \gtrsim d^{3}\log d$) is large enough. Within the proofs, we use auxiliary variables $\alpha, \beta, x, y$, that are sometimes redefined throughout the text.

Before we proceed, let us remark that $(X, Y)$ distributed according to \eqref{eq:bad-distribution} satisfies $\|X\| \le 1$ almost surely and $\|Y\|_{L_{\infty}} \le 1$, thus $r=m=1$. Our first lemma, proved in Appendix~\ref{sec:proof-of-first-lemma}, provides an excess risk lower bound for any vector $\omega \in \R^{d}$, provided that $b$ is large enough.
\begin{Lemma}
  \label{lem:first_lemma}
  Suppose that $b \geq \sqrt{d}/2, d \geq 4$, and $(X, Y)$ is distributed according to \eqref{eq:bad-distribution}. Then, for any
  $\omega \in \R^{d}$ we have
  \[
    R(\omega) - R(\omega^{*}_b)
    \geq
    \frac{1}{2}
      d^{-3/2}
      \norm{\omega - \omega^{*}_b}^{2}
    \quad\text{and also}\quad
    \omega^{*}_b
     =\frac{\sqrt{d}-1}{2\sqrt{d} - 1}\cdot\mathbf{1}.
  \]
\end{Lemma}
Further, we define an unconstrained least squares solution as (dropping the superscript $\operatorname{ERM}$ in our notation):
\begin{equation}
  \label{eq:unconstrained-erm-solution}
  \what_{\infty} = \left(n\widehat{\Sigma}\right)^{-1}
              \Big(\sum_{i=1}^{n} X_{i}Y_{i}\Big).
\end{equation}

In the proof of Theorem~\ref{thm:erm-ridge-lower-bound} we work 
on the event where $\widehat{\Sigma}$ is invertible which will be shown to hold with sufficient probability. This ensures the uniqueness of $\what_{\infty}$ hence, we remark that the result of Theorem~\ref{thm:erm-ridge-lower-bound} holds for \emph{any} constrained least squares estimator. 
Our proof strategy is quite straightforward: using Lemma~\ref{lem:first_lemma} we show that the excess risk of
$\what_{\infty}$ is lower bounded by $cd^{3/2}/n$,
while for large enough $b$, $\what_{\infty}$ is also a least squares solution constrained to the ball of radius $b$. Before stating our next lemma we introduce some additional notation.
Let
\[
    I = \{i \in \{1, \ldots, n\} : X_{i} \neq \mathbf{1}/d\}
\]
denote the (random) subset of
data points whose covariates are not equal to $\mathbf{1}/d$. Denote
\begin{equation}
  \label{eq:bad-covariance-definition}
  A = \sum_{i \in I} X_{i}X_{i}^{\mathsf{T}}
  \quad\text{and hence}\quad
  \widehat{\Sigma} = \frac{1}{n}\left(
    \left(n - \abs{I}\right)d^{-2}\mathbf{1}\mathbf{1}^{\mathsf{T}}
    + A
  \right).
\end{equation}
Further, let $v, \zeta \in \R^{d}$ denote the (random) vectors such that
\begin{equation}
  \label{eq:c-and-zeta-definitions}
  v_{i} = A_{ii}\sqrt{d}
  \quad\text{and}\quad
  \zeta = v - \abs{I}d^{-1/2}\mathbf{1}.
\end{equation}
In words, the $i$-th entry of $v$ denotes the number of observations
in the set $I$ whose $i$-th entry is non-zero. Conditionally on the size of 
the set $I$, $\E v = \abs{I}d^{-1/2}\mathbf{1}$ and hence, $\zeta$ represents
the noise present in the counts vector~$v$.
We will repeatedly rely on the following identities, which can be shown
via a simple counting argument:
\begin{equation}
  \label{eq:A1-equals-v}
  A\mathbf{1} = v = \abs{I}d^{-1/2}\mathbf{1} + \zeta \quad \text{and} \quad \inr{\zeta, \mathbf{1}} = 0.
\end{equation}
The following lemma provides a sharp inequality for the norm
of $\what_{\infty}$ as well as an exact expression for the vector $\what_{\infty}$ itself. The proof is deferred to Appendix~\ref{sec:proof-of-erm-expressions-lemma}.
\begin{Lemma}
  \label{lemma:erm-expressions}
  Let $\what_{\infty}$ be defined by \eqref{eq:unconstrained-erm-solution}. The following two identities hold whenever the matrix $A$ defined in
  \eqref{eq:bad-covariance-definition} is invertible:
    \begin{equation}
    \label{eq:formulsforwinf}
    \what_{\infty}
    =
      \frac{
        d^{3/2}\abs{I}^{-1}
      }{
        (n-\abs{I})^{-1}d^{2} + \mathbf{1}^{\mathsf{T}}A^{-1}\mathbf{1}
      }\mathbf{1}
      -
      \frac{d^{3/2}\abs{I}^{-1}}{
        (n-\abs{I})^{-1}d^{2} + \mathbf{1}^{\mathsf{T}}A^{-1}\mathbf{1}
      }A^{-1}\zeta,
  \end{equation}
  and
    \[
    \norm{\what_{\infty}}^{2}
    \leq
    n^{2}d^{-2}\mathbf{1}^{\mathsf{T}}A^{-2}\mathbf{1}.
    \]
\end{Lemma}

Note that the first summand in \eqref{eq:formulsforwinf} as well as the vector $\omega^*_b$ are both proportional to $\mathbf{1}$. However, it will be shown later that the second summand in \eqref{eq:formulsforwinf}, which is proportional to $A^{-1}\zeta$, is almost orthogonal to $\mathbf{1}$. Combining this observation with the fact that
$\ip{\mathbf{1}}{\zeta} = 0$ will yield the desired lower bound via Lemma \ref{lem:first_lemma}, provided that the magnitude of the second term in \eqref{eq:formulsforwinf} is large enough.

Combining Lemmas~\ref{lem:first_lemma} and \ref{lemma:erm-expressions},
the excess risk of the unconstrained least squares solution $\what_{\infty}$
can be expressed in terms of the random quadratic form
$\mathbf{1}^{\mathsf{T}}A^{-1}\mathbf{1}$ and the random vector $A^{-1}\zeta$.
Also, the norm of $\what_{\infty}$ can be upper-bounded in terms
of $\mathbf{1}^{\mathsf{T}}A^{-2}\mathbf{1}$. The following result provides sharp bounds on all the random quantities that we need.
\begin{Lemma}
  \label{lemma:bounds-on-quadratic-forms}
  Suppose that $d$ is large enough and $n \gtrsim d^{3}\log d$.
  Then, the following results hold simultaneously,
  with probability at least $1/2$:
  \begin{enumerate}[(a)]
      \item $\abs{I} \sim nd^{-1/2}$;
      \item $\|\zeta\|^2 \sim n$;
      \item The matrix $A$ defined by \eqref{eq:bad-covariance-definition} is
            invertible;
      \item $\zeta^{\mathsf{T}}A^{-1}\zeta \gtrsim d^{3/2}$;
            \item $\mathbf{1}^{\mathsf{T}}A^{-1}\mathbf{1} \lesssim n^{-1}d^{2}$;
      \item $\mathbf{1}^{\mathsf{T}}A^{-2}\mathbf{1}  \lesssim n^{-2}d^{3}$.
  \end{enumerate}
\end{Lemma}
Before presenting the formal proof (see Appendix~\ref{sec:proof-of-lemma-on-quadratic-forms}), we discuss the intuition behind the proof of this lemma. First, observe that $(a)$ follows from the fact that $|I|$ is Binomially distributed with parameters $n$, $d^{-1/2}$. The magnitude of $\|\zeta\|^2$ follows from a direct computation of its expectation and variance. For large enough $n$, we expect that $A \approx \E A$. Assuming this, we may focus on $\E A$, which conditionally on the size of set $I$ has the following simple form:
\[
\E A = \abs{I} \left((d^{-1} - (d^{3/2} + d)^{-1})I_{d} + (d^{3/2} +
  d)^{-1}\mathbf{1}\mathbf{1}^{\mathsf{T}}\right).
\]
Observe that the eigenvector corresponding to the largest eigenvalue of $\E A$ is proportional to $\mathbf{1}$, and the remaining eigenvectors complement this direction and form an orthonormal basis. Moreover, the above expression for $\E A$ implies that $\lambda_1(\E A) \sim \abs{I}d^{-1/2}$ and $\lambda_j(\E A) \sim \abs{I}d^{-1}$ for $j = 2, \ldots, d$. Thus, $\E A$ is invertible and in particular, we have
\[
\mathbf{1}^{\mathsf{T}}(\E A)^{-1}\mathbf{1} = d/\lambda_1(\E A) \lesssim n^{-1}d^2, \quad \text{and} \quad \mathbf{1}^{\mathsf{T}}(\E A)^{-2}\mathbf{1} = d/(\lambda_1(\E A))^2 \lesssim n^{-2}d^3.
\]
Finally, since by \eqref{eq:A1-equals-v} we have $\inr{\zeta, \mathbf{1}} = 0$, the vector $\zeta$ is orthogonal to the first eigenvalue of $\E A$. Therefore,
\[
\zeta^{\mathsf{T}}(\E A)^{-1}\zeta \ge \|\zeta\|^2/\lambda_{2}(\E A) \gtrsim d^{3/2}.
\]
With the above lemmas at hand, we are ready to prove Theorem~\ref{thm:erm-ridge-lower-bound}.

\begin{proof}[Proof of Theorem~\ref{thm:erm-ridge-lower-bound}]
  We work on the event of Lemma~\ref{lemma:bounds-on-quadratic-forms}.
  First, note that combining Lemmas~\ref{lemma:erm-expressions} and
  \ref{lemma:bounds-on-quadratic-forms} we have
  \[
   \norm{\what_{\infty}}^{2} \lesssim d.
  \]
  Thus, on the event of Lemma~\ref{lemma:bounds-on-quadratic-forms}, the
  unconstrained ERM solution $\what_{\infty}$ is also a solution over
  the Euclidean ball of any radius $b$ that satisfies $b \geq c\sqrt{d}$, where
  $c$ is some absolute constant.

  We will now lower bound the expected excess risk of $\what_{\infty}$. Observe that for any vector $x$ and a unit vector $u$ we have $\|x\| \ge |\inr{x, u}|$. Consider the unit vector $u = \zeta / \|\zeta\|$. Denote
  \[
  \beta = \frac{d^{3/2}\abs{I}^{-1}}{
        (n-\abs{I})^{-1}d^{2} + \mathbf{1}^{\mathsf{T}}A^{-1}\mathbf{1}}.
  \]
  Combining Lemmas~\ref{lem:first_lemma} and \ref{lemma:erm-expressions} together with $\inr{\zeta, \mathbf{1}} = 0$ given by \eqref{eq:A1-equals-v} we have 
  \begin{align}
R(\what_{\infty}) - R(\omega^{*}_{b}) &\geq
    \frac{1}{2}d^{-3/2}
    \bigg\|\left(
	\beta- \frac{\sqrt{d} - 1}{2\sqrt{d} - 1}\right)
      \mathbf{1}
      -
      \beta
      A^{-1}\zeta
     \bigg\|^{2} \\
   &\geq
    \frac{1}{2}d^{-3/2}
    \left(\inr{\frac{\zeta}{\|\zeta\|},\left(
	\beta- \frac{\sqrt{d} - 1}{2\sqrt{d} - 1}\right)
      \mathbf{1}
      -
      \beta
      A^{-1}\zeta}\right)^{2}
      \\
      &= \frac{1}{2}d^{-3/2}
    \left(\inr{\frac{\zeta}{\|\zeta\|},
      \beta
      A^{-1}\zeta}\right)^{2}.
    \label{eq:excess-risk-lower-bound-alpha-beta}
  \end{align}
  By Lemma \eqref{lemma:bounds-on-quadratic-forms} we have $\beta \gtrsim 1$, with probability at least $\frac{1}{2}$. Hence, the lower bound \eqref{eq:excess-risk-lower-bound-alpha-beta} implies on the event of Lemma \ref{lemma:bounds-on-quadratic-forms} that
  \[
  R(\what_{\infty}) - R(\omega^{*}_{b}) \ge \frac{1}{2}d^{-3/2}\beta^2\left(\frac{\zeta^{\mathsf{T}} A^{-1} \zeta}{\|\zeta\|}\right)^2 \gtrsim \frac{d^{3/2}}{n}.
  \]
  Since the event of Lemma \ref{lemma:bounds-on-quadratic-forms} holds with probability at least $\frac{1}{2}$, it follows that $\E  R(\what_{\infty}) - R(\omega^{*}_{b}) \gtrsim \frac{d^{3/2}}{n}$. This concludes the proof of our theorem.
\end{proof}
\section*{Acknowledgements}
We are indebted to Shahar Mendelson for fruitful discussions and valuable feedback: in particular, for suggesting us the technique to analyze the quadratic process in Theorem \ref{thm:main_thm} and for motivating us to work on the lower bounds. We are also grateful to Manfred Warmuth for providing a reference to the predictor that removes the excess logarithmic factor appearing in one of our results. Finally, we thank Jaouad Mourtada for many related discussions.

This work was conducted when Nikita Zhivotovskiy was at Google Research, Z\"{u}rich. Tomas Va\v{s}kevi\v{c}ius is supported by the EPSRC and MRC through the OxWaSP CDT programme (EP/L016710/1). 

\bibliography{mybib} 
\begin{appendix}


\section{Proof of Theorem \ref{thm:main_thm_second}}
\label{sec:secondmainthm}
Our analysis is based on the notion of average stability (for more details we refer to e.g., \citep*{shalev2014understanding} and reference therein). Although there is a vast literature on stability based techniques, with several recent related papers \citep*{koren2015fast, gonen2017average}, we could not find any general result that implies Theorem \ref{thm:main_thm_second}. Therefore, we provide an elementary proof. One of the differences compared with some of the previous results is that we provide an exact formula for the average stability as well as an exact analysis of the fitting term as an intermediate step in our proof. Our analysis of the fitting term exploits the curvature of the quadratic loss and, in particular, allows to rewrite the stability term as the multiplier term that appears in Theorem~\ref{thm:main_thm}. Crucially, stability based approach allows us to replace the quadratic term appearing in the proof of Theorem \ref{thm:main_thm} by a bias term that scales as $\frac{\lambda b^2}{n}$ and consequently, removes multiplicative factor $\log(\min\{n, d\})$ that is unimprovable for constrained ERM as we demonstrate in Proposition~\ref{prop:lowerbound}.

Similarly to the proof of Theorem~\ref{thm:main_thm} presented in Section~\ref{section:erm-proof}, we drop the subscript $b$ from $\omega^{*}_{b}$ and denote it by $\omega^{*}$ in the rest of this section.
We also introduce an additional independent element $(X_{n + 1}, Y_{n + 1})$ distributed according to $P_{r}$ to the sample. With a slight abuse of notation, we define for $j = 1, \ldots, n + 1$ the penalized (but not normalized by the sample size as in \eqref{eq:samplecov}) empirical second moment matrices by
\[
\tilde{\Sigma}_{\lambda} = \lambda I_d + \sum\limits_{i = 1}^{n + 1}X_iX_i^{\mathsf{T}} \quad \text{and} \quad \tilde{\Sigma}_{\lambda}^{(-j)} = \lambda I_d + \sum\limits_{i = 1, i \neq j}^{n + 1}X_iX_i^{\mathsf{T}}.
\]
For any $\lambda > 0$, the (unique) ridge estimator \eqref{eq:ridge} constructed on all but the $j$-th sample, and all the $n+1$ samples respectively, is defined as follows:
 \begin{equation}
     \label{eq:twodefs}
    \widehat{\omega}_{\lambda}^{(-j)} = (\tilde{\Sigma}^{(-j)}_{\lambda})^{-1}\sum_{i=1, i \neq j}^{n + 1} Y_{i}X_{i} \quad \text{and} \quad \tilde{\omega}_{\lambda} = \tilde{\Sigma}_{\lambda}^{-1}\sum_{i=1}^{n + 1} Y_{i}X_{i},
\end{equation}
Therefore, the ridge estimator defined in \eqref{eq:ridge} is trained on the first $n$ samples and hence it satisfies
\[
    \widehat{\omega}_{\lambda} = \widehat{\omega}_{\lambda}^{(-(n + 1))}.
\]
Let $\E$ denote the expectation with respect to all $n+1$ samples $(X_1, Y_1), \ldots, (X_{n + 1}, Y_{n +1})$. Since the sample is exchangeable we have
\begin{align*}
&\E R(\widehat{\omega}_{\lambda}) - R(\omega^*) =  \E(Y_{n + 1} - \langle \widehat{\omega}_{\lambda}, X_{n + 1}\rangle)^2 - \frac{1}{n + 1}\sum\limits_{i = 1}^{n + 1}\E(Y_{i} - \langle \omega^*, X_{i}\rangle)^2
\\
&\quad= \frac{1}{n + 1}\Bigl(\sum\limits_{i = 1}^{n + 1}\underbrace{\E\left((Y_{i} - \langle \widehat{\omega}_{\lambda}^{(-i)}, X_{i}\rangle)^2 -(Y_{i} - \langle \tilde{\omega}_{\lambda}, X\rangle)^2\right)}_{\text{Average stability}} 
\\
&\quad\quad\quad\quad+ \E\underbrace{\sum\limits_{i = 1}^{n + 1}\left((Y_{i} - \langle \tilde{\omega}_{\lambda}, X_{i}\rangle)^2 - (Y_{i} - \langle \omega^*, X_i\rangle)^2\right)}_{\text{Fitting term}}\Bigr),
\end{align*}
where the two terms in the last display are interpreted as the \emph{fitting-stability trade-off} \citep*{shalev2014understanding}. Indeed, the first term corresponds to the average sensitivity of the estimator to the perturbation in one point of the sample and is called the \emph{average stability}. The second term shows how the empirical loss of the estimator compares to the empirical loss of the best linear predictor $\omega^*$ and is called the \emph{fitting term}. The remainder of the proof is devoted to the analysis of these two terms.

\myparagraph{Average stability term}
We provide exact computations for the average stability term via the Sherman-Morrison formula (see e.g., \citep{hager1989updating}), which states that for any $j = 1, \ldots, n + 1$ we have
\begin{equation}
\label{eq:sherman-morrison}
 \left(\tilde{\Sigma}_{\lambda}^{(-j)}\right)^{-1} =
    \tilde{\Sigma}_{\lambda}^{-1}
    + \frac{
      \tilde{\Sigma}_{\lambda}^{-1}
      X_{j}
      X_{j}^{\mathsf{T}}
       \tilde{\Sigma}_{\lambda}^{-1}
    }{
      1 - X_j^{\mathsf{T}}\tilde{\Sigma}_{\lambda}^{-1}X_j
    }.
\end{equation}
To simplify the notation we denote for the rest of the proof the $j$-th (random) leverage score by
\begin{equation}
\label{eq:levscore}
h_j = X_j^{\mathsf{T}}\tilde{\Sigma}_{\lambda}^{-1}X_j.
\end{equation}
Using the definition \eqref{eq:twodefs}, Sherman-Morrison formula \eqref{eq:sherman-morrison} and simple algebra we demonstrate the following \emph{equality}:
\begin{align*}
&(Y_{i} - \langle\widehat{\omega}_{\lambda}^{(-i)}, X_i\rangle)^2 
\\
&= \left(Y_{i} - \left\langle (\tilde{\Sigma}^{(-i)}_{\lambda})^{-1}\sum_{j = 1}^{n + 1} Y_{j}X_{j}, X_i\right\rangle + \left\langle (\tilde{\Sigma}^{(-i)}_{\lambda})^{-1}Y_{i}X_{i}, X_i\right\rangle\right)^2
\\
&= \left(Y_{i} - \left\langle (\tilde{\Sigma}^{(-i)}_{\lambda})^{-1}\sum_{j = 1}^{n + 1} Y_{j}X_{j}, X_i\right\rangle + Y_i\left(h_i + \frac{h_i^2}{1 - h_i}\right)\right)^2
\\
&= \left(Y_{i} - \left\langle \tilde{\Sigma}^{-1}_{\lambda}\sum_{j = 1}^{n + 1} Y_{j}X_{j}, X_i\right\rangle - \frac{1}{1 - h_i}\left\langle \tilde{\Sigma}^{-1}_{\lambda}X_iX_i^{\mathsf{T}}\tilde{\omega}_{\lambda}, X_i\right\rangle + \frac{Y_ih_i}{1 - h_i}\right)^2
\\
&= \left(Y_{i} - \langle\tilde{\omega}_{\lambda}, X_i\rangle - \frac{h_i}{1 - h_i}\left\langle \tilde{\omega}_{\lambda}, X_i\right\rangle + \frac{Y_ih_i}{1 - h_i}\right)^2 = \left(\frac{1}{1 - h_i}\right)^2\left(Y_{i} - \langle\tilde{\omega}_{\lambda}, X_i\rangle\right)^2.
\end{align*}
The above result implies
\begin{align}
    \E\left((Y_{i} - \langle \widehat{\omega}_{\lambda}^{(-i)}, X_{i}\rangle)^2 -(Y_{i} - \langle \tilde{\omega}_{\lambda}, X_i\rangle)^2\right) &= \E\left(\left(\frac{1}{1 - h_i}\right)^2 -1\right)(Y_{i} - \langle \tilde{\omega}_{\lambda}, X_i\rangle)^2
    \\
    \label{eq:stability-exact}
    &=  \E\left(\frac{h_{i}}{(1-h_{i})^{2}} + \frac{h_{i}}{1 - h_{i}}\right)(Y_{i} - \langle \tilde{\omega}_{\lambda}, X_i\rangle)^2.
\end{align}
Finally, we show that if $\lambda > cr^2$ for some numerical constant $c > 0$, then $h_i$ is separated from $1$. First, observe that $X_jX_j^{\mathsf{T}} + \lambda I_d \preceq \tilde{\Sigma}_{\lambda} $. Since $X_jX_j^{\mathsf{T}}$ is a rank one matrix having at most one non-zero positive eigenvalue denoted by $\mu$, we have
\begin{align}
\label{eq:hj}
0 &\le h_j = X_j^{\mathsf{T}}\tilde{\Sigma}_{\lambda}^{-1}X_j = \Tr(\tilde{\Sigma}_{\lambda}^{-1}X_jX_j^{\mathsf{T}}) \nonumber
\\
&\le \Tr((X_jX_j^{\mathsf{T}} + \lambda I_d)^{-1}X_jX_j^{\mathsf{T}}) = \frac{\mu}{\mu + \lambda} \le \frac{r^2}{r^2 + \lambda} \le \frac{1}{1 + c} ,
\end{align}
where we applied the facts that $(X_jX_j^{\mathsf{T}} + \lambda I_d)^{-1}X_jX_j^{\mathsf{T}}$ is rank one matrix, $x\mapsto \frac{x}{x + \lambda}$ is monotone, and $\mu \le r^2$. Using simple algebra and the identity given in \eqref{eq:stability-exact}, for any $\lambda \ge cr^2$ we have
\begin{equation}
\label{eq:averagestability}
\E\left((Y_{i} - \langle \widehat{\omega}_{\lambda}^{(-i)}, X_{i}\rangle)^2 -(Y_{i} - \langle \tilde{\omega}_{\lambda}, X_i\rangle)^2\right) \le \frac{1 + 3c + 2c^2}{c^2}\E h_i(Y_{i} - \langle \tilde{\omega}_{\lambda}, X_i\rangle)^2.
\end{equation}
\myparagraph{Fitting term}
In contrast to the naive upper bound, which follows by adding and subtracting $\frac{1}{n+1}(\lambda\norm{\tilde{\omega}_{\lambda}}^{2} - \norm{\omega^*}^{2})$ to the (unnormalized) fitting term and using the fact that $\tilde{\omega}_{\lambda}$ minimizes the ridge regression optimization objective:
\[
  \E \sum\limits_{i = 1}^{n + 1}\left((Y_{i} - \langle \tilde{\omega}_{\lambda}, X_{i}\rangle)^2 - (Y_{i} - \langle \omega^*, X_i\rangle)^2\right)
  \leq \lambda\left(\norm{\omega^{*}}^{2} - \norm{\tilde{\omega}_{\lambda}^{2}}\right)
  \leq \lambda b^{2},
\]
in the proof below, we exploit the curvature of the squared loss, which results in an improved upper bound. The improvement that comes from extra negative terms allow us to compensate rewrite the average stability term as localized multiplier term, thus establishing a direct link between localization and stability analysis. Recall the definition \eqref{eq:omegaandxi} of the noise variable $\xi$. We show that the following deterministic inequality holds for the fitting term:
\begin{equation}
    \label{eq:geometricproperty}
    \sum\limits_{i = 1}^{n + 1}\left((Y_{i} - \langle \tilde{\omega}_{\lambda}^{}, X_{i}\rangle)^2 - (Y_{i} - \langle \omega^*, X_i\rangle)^2\right) \le -\sum\limits_{i = 1}^{n + 1}(\xi_i - \widehat{\xi}_i)^2 + \lambda\|\omega^*\|^2/2,
\end{equation}
where we denote $\widehat{\xi}_i = Y_i - \langle \tilde{\omega}_{\lambda}, X_i\rangle$. Since $\tilde{\omega}_{\lambda}$ by the first order optimality conditions nullifies the gradient of the penalized empirical risk, that is
$$
  0 = \nabla_{\tilde{\omega}_{\lambda}} \left(\sum_{i=1}^{n+1} (Y_{i} - \ip{\tilde{\omega}_{\lambda}}{X_{i}})^{2} + \lambda\norm{\tilde{\omega}_{\lambda}}^{2}\right)
$$
it follows that
\[
\sum\limits_{i = 1}^{n + 1}X_i(\langle\tilde{\omega}_{\lambda}, X_i\rangle- Y_i)+\lambda \tilde{\omega}_{\lambda} = 0.
\]
Taking an inner product with $\tilde{\omega}_{\lambda} - \omega^*$ in the above equality yields
\[
\sum\limits_{i = 1}^{n + 1}(\langle\tilde{\omega}_{\lambda}, X_i\rangle- Y_i)(\langle\tilde{\omega}_{\lambda}, X_i\rangle - \langle\omega^*, X_i\rangle) =  \lambda(\langle \tilde{\omega}_{\lambda} , \omega^*\rangle - \|\tilde{\omega}_{\lambda}\|^2).
\]
The bound \eqref{eq:geometricproperty} follows by combining the last equality together with the following formula for the excess loss $P_n {\cal L}_{\tilde{\omega}_{\lambda}}$ as in the proof of Theorem \ref{thm:main_thm} (cf.\ Equation~\eqref{eq:excess-loss-decomposition}):
\begin{align*}
&\sum\limits_{i = 1}^{n + 1}\left((Y_{i} - \langle \tilde{\omega}_{\lambda}, X_{i}\rangle)^2 - (Y_{i} - \langle \omega^*, X_i\rangle)^2\right) 
\\
&= -\sum\limits_{i = 1}^{n + 1}\inr{\tilde{\omega}_{\lambda} - \omega^*,X_i}^2 + 2\sum\limits_{i = 1}^{n + 1}(\inr{\tilde{\omega}_{\lambda},X_i}-Y_i)\cdot \inr{\tilde{\omega}_{\lambda}-\omega^*,X_i}
\\
&= -\sum\limits_{i = 1}^{n + 1}(\xi_i - \widehat{\xi}_i)^2 + 2\lambda(\langle \tilde{\omega}_{\lambda} , \omega^*\rangle - \|\tilde{\omega}_{\lambda}\|^2)
\\
&\le -\sum\limits_{i = 1}^{n + 1}(\xi_i - \widehat{\xi}_i)^2 + \lambda\|\omega^*\|^2/2,
\end{align*}
where in the last line we apply the Cauchy-Schwarz inequality and the fact that $\|\tilde{\omega}_{\lambda}\|\|\omega^*\| - \|\tilde{\omega}_{\lambda}\|^2 \le \|\omega^*\|^2/4$.

\myparagraph{Completing the proof} We are ready to finish the proof of Theorem \ref{thm:main_thm_second} via the stability-fitting trade-off and inequalities \eqref{eq:averagestability} and \eqref{eq:geometricproperty}. Indeed, optimizing the below quadratic function yields the following inequality:
\[
\alpha\widehat{\xi}_i^2 - (\xi_i - \widehat{\xi}_i)^2 = \widehat{\xi}_i^2(\alpha - 1) + 2\widehat{\xi}_i\xi_i - \xi_i^2 \le \xi_i^2\left(\frac{\alpha}{1 - \alpha}\right),
\]
which holds for any $0 < \alpha < 1$, we have that if $\frac{1 + 3c + 2c^2}{c^2}h_i < 1$, then
\begin{align*}
&\E R(\widehat{\omega}_{\lambda}) - R(\omega^*) 
\\
&\le \frac{1}{n + 1}\E\left(\sum\limits_{i = 1}^{n + 1}\left(\frac{1 + 3c + 2c^2}{c^2}h_i\widehat{\xi}_i^2 -(\xi_i - \widehat{\xi}_i)^2\right) + \lambda\|\omega^*\|^2/2\right)
\\
&\le \frac{1}{n + 1}\E\left(\sum\limits_{i = 1}^{n + 1}\left(\frac{1 + 3c + 2c^2}{c^2}h_i{\xi}_i^2\right)/\left(1 -  \frac{1 + 3c + 2c^2}{c^2}h_i\right) + \lambda\|\omega^*\|^2/2\right).
\end{align*}
Using the inequality on the leverage scores given in \eqref{eq:hj}, a simple computation shows that the choice $c = 3$  guarantees that $\frac{1 + 3c + 2c^2}{c^2}/\left(1 -  \frac{1 + 3c + 2c^2}{c^2}h_i\right) \le 14$, which concludes our proof.

Observe that we never used any specific properties of $\omega^*$ and it can be replaced by $\emph{any}$ vector in $\R^{d}$. In contrast, our proof technique based on localization (cf. Section~\ref{section:erm-proof}) crucially relies on the fact that $\omega^{*} = \omega^{*}_{b}$ minimizes the population risk over all vectors in $\mathcal{W}_{b}$. Finally, note that we have established our results for the leverage scores $h_i$ that are computed on the sample of size $n + 1$. However, it is easy to see that one may decrease the sample size by one so that the original claim of Theorem \ref{thm:main_thm_second} holds.
\qed

\section{Proof of Proposition~\ref{prop:lowerbound}}
\label{sec:prooflowerbound}
Without loss of generality we set $r = 1$ (otherwise we can rescale the covariates introduced below by $r$).
Let $Y = 0$, so that $\omega^*_b = 0$ and $Y = \inr{\omega^*_b, X}$ (i.e., the problem is noise-free). Let $k = k(d, n) \le d$ be an integer to be specified later. Let the covariates $X$ follow a uniform distribution on the set of first $k$ basis vectors $e_{1}, \ldots, e_{k}$. The distribution of the samples $(X, Y)$ hence satisfies $\|X\| \le 1$.

We aim to choose the value of $k$ such that with probability at least $1/2$, at most $k - 1$ out of the $k$ basis vectors are observed in the random sample $X_1, \ldots, X_n$ (i.e., at least one of the basis vectors is not observed). This analysis follows from the coupon collector argument showing that one needs, with probability at least $1/2$, a sample of size at least $ck\log k$ to observe all $k$ vectors where $c$ is a numerical constant. Let $T$ be a random variable that counts the (random) number of trials needed to observe all $k$ basis vectors.  A basic result (see, e.g., \citep*[Section 3.6]{motwani2010randomized}) shows that for any $t > 0$,
\[
\Pr\left(\left|T - kH_{k}\right| \ge tk \right) \le \frac{\pi^2}{6t^2},
\]
where $H_{k}$ is the $k$-th Harmonic number. Since $\log k < H_{k}$ we have, with probability at least $1/2$,
\[
T > k \log k - \frac{\pi}{\sqrt{3}}k \ge \frac{k}{2}\log k,
\]
where a simple computation shows that the last inequality holds provided that $k \ge 38$. In what follows, we choose $k$ to be the smallest integer such that $n \leq \frac{1}{2}k \log k$; the condition $k \ge 38$ can be always be satisfied provided that $n$ is large enough. By the above, with probability at least $1/2$ there is at least one basis vector among $e_{1}, \ldots, e_{k}$ such that it is not included in the sample $X_{1}, \ldots, X_n$. Denote such a (random) basis vector by $e^*$ and the corresponding event by $E$. Observe that on this event the vector $be^* \in \mathcal{W}_{b}$ satisfies $\inr{be^*, X_i} = 0$;  thus it is \emph{one of} the least squares solutions (on the observed sample) in $\mathcal{W}_{b}$ with $R(be^*) = b^2/k$. Let $\widehat{\omega}$ be equal to $be^*$ on $E$ and equal to any linear least squares on the complementary event $\overline{E}$. Using Markov's inequality we have
\[
\E R(\widehat{\omega}) - R(\omega^*_b) = \E R(\widehat{\omega}) \ge \frac{b^2}{k}\Pr\left(R(\widehat{\omega}) \ge \frac{b^2}{k}\right) \ge \frac{b^2}{k}\Pr\left(E\right) \ge \frac{b^2}{2k} \ge \frac{b^2\log n}{10n},
\]
where the last step can be proved as follows: by our choice of $k$, we have
$\frac{1}{2}(k - 1)\log (k - 1) \leq n \leq \frac{1}{2}k\log k \le (k - 1)^2$ which implies $k \leq \frac{2n}{\log (k - 1)} + 1\leq \frac{4n}{\log n} + 1 \le \frac{5n}{\log n}$. This concludes our proof.
\qed

\section{Proof of Proposition~\ref{prop:l4l2}}
\label{sec:proof-of-l4l2}
Let $c \ge 1$ denote the numerical constant that satisfies the assumption $\E\inr{\omega, X}^4 \le c\left(\E\inr{\omega, X}^2\right)^2$.
Since the leverage scores $X^{\mathsf{T}}(n\widehat{\Sigma}_{\lambda})^{-1}X$ are in $[0,1]$,
using $\|\xi\|_{L_\infty} \le m + rb$ we have
\[
\xi^2X^\mathsf{T}\widehat{\Sigma}_{\lambda}^{-1}X \le n(m + rb)^2.
\]
Let $E$ denote the event that $\|\Sigma^{1/2}\widehat{\Sigma}_{\lambda}^{-1}\Sigma^{1/2}\| \le 2$. Combining the above inequality, the Cauchy-Schwarz inequality, and the $L_4$--$L_2$ moment equivalence assumption $\|\xi\|_{L_4}\lesssim \|\xi\|_{L_2} = \sqrt{R(\omega^*_b)} \leq m$, we have 
\begin{align}
    \E\xi^2X^\mathsf{T}\widehat{\Sigma}_{\lambda}^{-1}X
    &=  \E\xi^2 (X^\mathsf{T}\Sigma^{-1/2})(\Sigma^{1/2}\widehat{\Sigma}_{\lambda}^{-1}\Sigma^{1/2})(\Sigma^{-1/2}X)\ind_{E} + \E\xi^2X^\mathsf{T}\widehat{\Sigma}_{\lambda}^{-1}X\ind_{\overline{E}}
    \\
    &\le \E\xi^2\|\Sigma^{-1/2}X\|^2
    \|\Sigma^{1/2}\widehat{\Sigma}_{\lambda}^{-1}\Sigma^{1/2}\|\ind_{E} + n(m + rb)^2\Pr(\overline{E})
    \\
    &\le 2\E\xi^2\|\Sigma^{-1/2}X\|^2 + n(m + rb)^2\Pr(\overline{E})
    \\
    &\le 2\sqrt{\E\xi^4}\sqrt{\E\|\Sigma^{-1/2}X\|^4} + n(m + rb)^2\Pr(\overline{E})
     \\
    \label{eq:l2-l4-proof-intermediate-step}
    &\lesssim 2m^2\sqrt{\E\|\Sigma^{-1/2}X\|^4} + n(m + rb)^2\Pr(\overline{E}).
\end{align}
A direct calculation \citep[Remark 3]{mourtada2019exact} shows that $\E\|\Sigma^{-1/2}X\|^4 \lesssim d^2$ under our assumption. Indeed, fixing $\omega_i = \Sigma^{-1/2} e_i$ the following holds for $i = 1, \ldots, d$:
\[
\E\inr{\Sigma^{-1/2}e_i, X}^4 \le c\left(\E\inr{\Sigma^{-1/2}e_i, X}^2\right)^2 = c.
\]
Therefore, by the Cauchy-Schwarz inequality we have
\begin{align}
    \E\|\Sigma^{-1/2}X\|^4 &= \E\left(\sum\limits_{i = 1}^d\inr{e_i, \Sigma^{-1/2}X}^{2}\right)^2 \nonumber
    \\
    &\le \sum\limits_{i, j}^d\sqrt{\E\left(\inr{e_i, \Sigma^{-1/2}X}\right)^4}\sqrt{\E\left(\inr{e_j, \Sigma^{-1/2}X}\right)^4} \le cd^2. \label{eq:l2-l4-marginals-d-bound}
\end{align}
Under the $L_{4}$--$L_{2}$ moment equivalence assumption on the marginals $\ip{\omega}{X}$, the following lower tail bound given in \citep[Theorem 1.1]{oliveira2016lower} shows that for any $\delta \in (0, 1)$, with probability at least $1 - \delta$, simultaneously for all $v\in \mathbb{R}^d$, it holds that
\begin{equation}
\label{eq:olivreiabound}
v^\mathsf{T}\widehat{\Sigma}_{0}v \ge \left(1 - 9c\sqrt{\frac{d + 2\log \frac{2}{\delta}}{n}}\right)v^\mathsf{T}\Sigma v.
\end{equation}
For $\delta = \frac{1}{n}$ we have $1 - 9c\sqrt{\frac{d + 2\log (2n)}{n}} \ge \frac{1}{2}$, provided that $n \gtrsim d$. Hence, by \eqref{eq:olivreiabound}, with probability at least $1 - \frac{1}{n}$,
\begin{align*}
    \|\Sigma^{1/2}\widehat{\Sigma}_{\lambda}^{-1}\Sigma^{1/2}\|
    &= \left(\inf\limits_{v \in S^{d - 1}}v^\mathsf{T}\Sigma^{-1/2}\widehat{\Sigma}_{\lambda}\Sigma^{-1/2}v\right)^{-1} 
    \\
    &= \left(\inf\limits_{u: u^\mathsf{T}\Sigma u = 1}u^\mathsf{T}\widehat{\Sigma}_{\lambda}u\right)^{-1} \le \left(\inf\limits_{u: u^\mathsf{T}\Sigma u = 1}u^\mathsf{T}\widehat{\Sigma}_{0}u\right)^{-1} \le 2.
\end{align*}
The above inequality implies $\Pr(\overline{E}) \le \frac{1}{n}$ and  combined with inequalities \eqref{eq:l2-l4-proof-intermediate-step} and \eqref{eq:l2-l4-marginals-d-bound} yields
\[
    \E\xi^2X^\mathsf{T}\widehat{\Sigma}_{\lambda}^{-1}X \lesssim dm^2 + (m + rb)^2 \lesssim dm^{2} + r^{2}b^{2}.
\]
The proof is complete.

We remark that one may choose $\delta = \exp(-c_0n)$ in \eqref{eq:olivreiabound}, for some small enough numerical constant $c_0$ and improve the resulting bound. This may be important since as discussed in Section~\ref{sec:corollaries}, a bound on the quadratic term better than $\frac{r^2b^2}{n}$ could be possible under moment equivalence assumptions. Since this is not the main focus of our paper we do not pursue this direction.
\qed

\section{Proofs of Lemmas Supporting Theorem~\ref{thm:erm-ridge-lower-bound}}
\subsection{Proof of Lemma~\ref{lem:first_lemma}}
\label{sec:proof-of-first-lemma}
The proof is split into two steps.
We first compute $\omega^{*}_{\infty} = \inf\limits_{\omega \in \R^{d}} R(\omega)$
and show that $\norm{\omega^{*}_{\infty}} \leq \sqrt{d}/2$
so that $\omega_{\infty}^{*} = \omega^{*}_b$ whenever $b \geq \sqrt{d}/2$.
Next, we show that the lower bound follows via the Bernstein assumption
\eqref{eq:bernstein_eq}.

\myparagraph{Computing $\omega^{*}_b$}

Differentiating $R(\omega)$ with respect to $\omega$ and
applying the first order optimality conditions, we obtain the
following well-known expression for an unconstrained minimizer of the
population risk over $\R^{d}$: $\omega^{*}_{\infty} = \Sigma^{-1}\E XY$,
where $\Sigma = \E XX^{\mathsf{T}}$.
A simple calculation shows that
\begin{align}
  \label{eq:alpha-beta-definition}
  \Sigma = \alpha \mathbf{1}\mathbf{1}^{\mathsf{T}} + \beta I_{d},
  &\quad\text{with}\quad
  \alpha = (1 - d^{-1/2})d^{-2} + (d^{2} + d^{3/2})^{-1}
  \\
  &\quad\text{and}\quad
  \beta = d^{-3/2} - (d^{2} + d^{3/2})^{-1}.
\end{align}
By the Sherman-Morrison formula, we have
$$
  \Sigma^{-1}
  =
  \left(\beta I_{d}\right)^{-1}
  - \frac{
  \left(\beta I_{d}\right)^{-1}
  \alpha\mathbf{1}\mathbf{1}^{\mathsf{T}}
  \left(\beta I_{d}\right)^{-1}
  }{
    1 + \alpha\mathbf{1}^\mathsf{T}
    \left(\beta I_{d}\right)^{-1}
    \mathbf{1}
  }
  =
  \beta^{-1}I_{d}
  - \frac{\alpha\beta^{-2}}{1 + \alpha\beta^{-1}d}
    \mathbf{1}\mathbf{1}^\mathsf{T},
$$
which plugged into the equation $\omega^{*}_{\infty} = \Sigma^{-1}\E XY$ yields
\begin{align}
  \label{eq:w-star-solution}
  \omega^{*}_b
  =
  \left(\beta^{-1} - \frac{\alpha\beta^{-2}d}{1 + \alpha\beta^{-1}d}\right)
  (1-d^{-1/2})d^{-1}\cdot\mathbf{1}
  =\frac{\sqrt{d}-1}{2\sqrt{d} - 1}\cdot\mathbf{1}.
\end{align}
For all $d \geq 1$ we have $0 \leq (\sqrt{d} - 1)/ (2\sqrt{d} - 1) \leq 1/2$
and, in particular, $\norm{\omega^{*}_{\infty}} \leq \sqrt{d}/2 \leq b$.

\myparagraph{Lower bounding the excess risk}
Let $\omega$ denote any parameter vector in $\R^{d}$.
Since we have already shown that $\omega^{*}_{b}$ minimizes $R(\omega)$ over
all of $\R^{d}$, by the Bernstein assumption stated in \eqref{eq:bernstein_eq}
we have
\begin{align*}
  R(\omega) - R(\omega^{*}_{b})
  &\geq \E \ip{X}{\omega - \omega^{*}_{b}}^{2}
  = (\omega - \omega^{*}_{b})^{\mathsf{T}} \Sigma (\omega - \omega^{*}_{b})
  \\
  &= (\omega - \omega^{*}_{b}) (\alpha\mathbf{1}\mathbf{1}^{\mathsf{T}} + \beta
  I_{d})(\omega - \omega^{*}_{b}),
\end{align*}
with the values of $\alpha$ and $\beta$ given in~\eqref{eq:alpha-beta-definition}.
Since $\mathbf{1}\mathbf{1}^{\mathsf{T}}$ is positive semi-definite,
it hence follows that
$$
  R(\omega) - R(\omega^{*}_{b})
  \geq \beta \norm{\omega - \omega^{*}_{b}}^{2}.
$$
Finally, for $d \geq 4$ we have $\beta \geq \frac{1}{2}d^{-3/2}$, which
completes our proof.
\qed

\subsection{Proof of Lemma~\ref{lemma:erm-expressions}}
\label{sec:proof-of-erm-expressions-lemma}
\myparagraph{Computing $\what_{\infty}$}

We set once again $\alpha = (n - \abs{I})d^{-2}$ and
$y = \mathbf{1}^{\mathsf{T}}A^{-1}\mathbf{1}$. Combining~\eqref{eq:unconstrained-erm-solution} and~\eqref{eq:bad-covariance-definition} with $\sum_{i=1}^{n}X_{i}Y_{i} = (n - \abs{I})\mathbf{1}/d$ and the Sherman-Morrison formula we have
\begin{equation}
  \label{eq:w_hat-sherman-morrison}
  \what_{\infty}
  = d\alpha \left(
    A^{-1}\mathbf{1} - \frac{\alpha y A^{-1}\mathbf{1}}{1 + \alpha y}
  \right)
  = \left(
    d\alpha - \frac{d\alpha^{2} y}{1 + \alpha y}
  \right)
  A^{-1}\mathbf{1}
  =
  \frac{d\alpha}{1 + \alpha y}
  A^{-1}\mathbf{1}.
\end{equation}
By~\eqref{eq:A1-equals-v}, we have $A\mathbf{1} =
\abs{I}d^{-1/2}\mathbf{1} + \zeta$. Multiplying both sides by $A^{-1}$ and
rearranging yields
\[
  A^{-1}\mathbf{1} = \abs{I}^{-1}d^{1/2}(\mathbf{1} - A^{-1}\zeta).
\]
Plugging the above into \eqref{eq:w_hat-sherman-morrison} yields
\begin{equation}
  \label{eq:w_hat-final-expression}
  \what_{\infty}
  =
  \frac{d^{3/2}\abs{I}^{-1}\alpha}{1 + \alpha y}
  (\mathbf{1} - A^{-1}\zeta).
\end{equation}

\myparagraph{Computing $\norm{\what_{\infty}}^{2}$}
Using the computations as above, we obtain
\begin{equation}
  \label{eq:what-norm-initial-equation}
  \norm{\what_\infty}^{2}
  = \ip{\what_\infty}{\what_\infty}
  = (n - \abs{I})^{2}d^{-2} \cdot
  \mathbf{1}^{\mathsf{T}} (n\widehat{\Sigma})^{-2} \mathbf{1}.
\end{equation}
To simplify the notation, let $\alpha = (n - \abs{I})d^{-2}$,
$x = \mathbf{1}^{\mathsf{T}}A^{-2}\mathbf{1}$
and $y = \mathbf{1}^{\mathsf{T}}A^{-1}\mathbf{1}$.
Applying the Sherman-Morrison formula together with
\eqref{eq:bad-covariance-definition} we have
\begin{align*}
  &\mathbf{1}^{\mathsf{T}} (n\widehat{\Sigma})^{-2} \mathbf{1}
  \\
  &=
  \mathbf{1}^{\mathsf{T}}
  \left(
   A^{-1}
   -
   \frac{\alpha A^{-1}\mathbf{1}\mathbf{1}^{\mathsf{T}}A^{-1}}
        {1 + \alpha y}
  \right)^{2}
  \mathbf{1} \\
  &=
  \mathbf{1}^{\mathsf{T}}
  \left(
   A^{-2}
   -
   \frac{\alpha A^{-2}\mathbf{1}\mathbf{1}^{\mathsf{T}}A^{-1}}
        {1 + \alpha y}
   -
   \frac{\alpha A^{-1}\mathbf{1}\mathbf{1}^{\mathsf{T}}A^{-2}}
        {1 + \alpha y}
   +
   \frac{\alpha^{2}
         A^{-1}\mathbf{1}\mathbf{1}^{\mathsf{T}}
         A^{-2}
         \mathbf{1}\mathbf{1}^{\mathsf{T}}A^{-1}
        }
        {(1 + \alpha y)^{2}}
  \right)
  \mathbf{1} \\
  &=
   x
   -
   \frac{\alpha x y}{1 + \alpha y}
   -
   \frac{\alpha y x} {1 + \alpha y}
   +
   \frac{\alpha^{2} yxy}{(1 + \alpha y)^{2}} =
  \frac{x}{(1 + \alpha y)^{2}}.
\end{align*}
Plugging the above into \eqref{eq:what-norm-initial-equation} yields
\[
  \norm{\what_\infty}^{2}
  =
    (n - \abs{I})^{2} d^{-2}
    \frac{\mathbf{1}^{\mathsf{T}}A^{-2}\mathbf{1}}
         {\left(1 + (n - \abs{I})d^{-2}
           \mathbf{1}^{\mathsf{T}}A^{-1}{\mathbf{1}}\right)^{2}}
  \leq
  n^{2}d^{-2}\mathbf{1}^{\mathsf{T}}A^{-2}\mathbf{1}.
\]
The claim follows. \qed

\subsection{Proof of Lemma~\ref{lemma:bounds-on-quadratic-forms}}
\label{sec:proof-of-lemma-on-quadratic-forms}
The proof is based on applying the union bound on the probability of several events. Adjusting the constants one may always guarantee that the statement of Lemma~\ref{lemma:bounds-on-quadratic-forms} holds with probability at least $\frac{1}{2}$. By writing that the event holds with \emph{sufficient probability} we mean that it holds with probability at least $\frac{99}{100}$. Additionally, in many places we work conditionally on the size of the set $|I|$.
\myparagraph{Controlling $\abs{I}$}
The result follows from Chebyshev's inequality since $\abs{I}$ follows the Binomial distribution with parameters $n, d^{-1/2}$.

\myparagraph{Bound on $\norm{\zeta}^{2}$}
Recalling \eqref{eq:bad-covariance-definition} and \eqref{eq:c-and-zeta-definitions} we may rewrite $v_i = \sum\limits_{j = 1}^{|I|}v_{i, j}$, where $v_{i, j}$ has a Bernoulli distribution with parameter $d^{-1/2}$. Moreover, for any fixed $i$ we have that $v_{i, 1}, \ldots, v_{i, |I|}$ are independent and for any $j$ it holds that $\sum\limits_{i = 1}^d v_{i, j} = d^{1/2}$.  Combining these facts we have
\begin{align*}
\norm{\zeta}^2 &= \sum\limits_{i = 1}^{d}\left(\sum\limits_{j = 1}^{|I|}\left(v_{i, j} - d^{-1/2}\right)\right)^2 
\\
&= \sum\limits_{i = 1}^{d}\sum\limits_{j = 1}^{|I|}\left(v_{i, j} - d^{-1/2}\right)^2 + \sum\limits_{i = 1}^{d}\sum\limits_{j\neq k}^{|I|}\left(v_{i, j} - d^{-1/2}\right)\left(v_{i, k} - d^{-1/2}\right)
\\
&=d^{1/2}|I| - |I| + \sum\limits_{i = 1}^{d}\sum\limits_{j\neq k}^{|I|}\left(v_{i, j} - d^{-1/2}\right)\left(v_{i, k} - d^{-1/2}\right).
\end{align*}
We proceed with analysis of the zero mean sum $\sum\limits_{i = 1}^{d}\sum\limits_{j\neq k}^{|I|}\left(v_{i, j} - d^{-1/2}\right)\left(v_{i, k} - d^{-1/2}\right)$.
Observe that for any given $j$ the values $v_{1, j}, \ldots, v_{d, j}$ are not independent but are sampled with replacement. However, it is possible to avoid this problem using a direct computation.
First, for $i_1 \neq i_2$ and any $j$ we have
\begin{align*}
\E \left(v_{i_1, j} - d^{-1/2}\right)\left(v_{i_2, j} - d^{-1/2}\right) &= \E v_{i_1, j}v_{i_2, j} - d^{-1} = \frac{d^{1/2}}{d}\cdot\frac{d^{1/2} - 1}{d - 1} - d^{-1}
\\
&= -\frac{1}{d^{3/2} + d}.
\end{align*}
This implies the following correlation identity
\begin{align*}
&\E\left(\sum\limits_{j\neq k}^{|I|}\left(v_{i_1, j} - d^{-1/2}\right)\left(v_{i_1, k} - d^{-1/2}\right)\right)\left(\sum\limits_{j\neq k}^{|I|}\left(v_{i_2, j} - d^{-1/2}\right)\left(v_{i_2, k} - d^{-1/2}\right)\right)
\\
&=\E\left(\sum\limits_{j\neq k}^{|I|}\left(v_{i_1, j} - d^{-1/2}\right)\left(v_{i_1, k} - d^{-1/2}\right)\left(v_{i_2, j} - d^{-1/2}\right)\left(v_{i_2, k} - d^{-1/2}\right)\right)
\\
&= \frac{|I|^2 - |I|}{(d^{3/2} + d)^2}.
\end{align*}
The last identity leads to
\begin{align*}
&\E\left(\sum\limits_{i = 1}^{d}\sum\limits_{j\neq k}^{|I|}\left(v_{i, j} - d^{-1/2}\right)\left(v_{i, k} - d^{-1/2}\right)\right)^2
\\
&=\sum\limits_{i = 1}^{d}\E\left(\sum\limits_{j\neq k}^{|I|}\left(v_{i, j} - d^{-1/2}\right)\left(v_{i, k} - d^{-1/2}\right)\right)^2 + \left(d^2 - d\right)\frac{|I|^2 - |I|}{(d^{3/2} + d)^2}
\\
&= d(|I|^2 - |I|)\left(d^{-1/2}(1 - d^{-1/2})\right)^2 + \left(d^2 - d\right)\frac{|I|^2 - |I|}{(d^{3/2} + d)^2} \le 2|I|^2.
\end{align*}
Finally, using Chebyshev's inequality we have $\norm{\zeta}^2 \sim \abs{I}d^{1/2} \sim n$ with sufficient probability.

\myparagraph{Invertibility of $A$}
Observe that $A$ is a sum $|I|$ independent positive semi-definite random matrixes such that each summand has the operator norm equaling one.
Using the lower tail of the matrix Chernoff bound \citep[Theorem 5.1.1]{tropp15} we have
\[
\Pr(\lambda_{d}(A) \le \lambda_d(\E A)/2) \le d(2e^{-1})^{\abs{I}(d^{-1} + (d^{3/2} + d)^{-1})/2},
\]
which is arbitrary small for large enough $d$ and $n \gtrsim d^{3}\log d$. Finally, observe that 
\[
\lambda_{d}(\E A) = \lambda_{d}\left(\abs{I} \left((d^{-1} - (d^{3/2} + d)^{-1})I_{d} + (d^{3/2} +
  d)^{-1}\mathbf{1}\mathbf{1}^{\mathsf{T}}\right)\right) \sim |I|d^{-1}.
\]
Therefore, $\lambda_{d}(A) > 0$ with sufficient probability and $A$ is invertible.
\myparagraph{A lower bound on $\lambda_{1}(A)$}
By \eqref{eq:c-and-zeta-definitions}
we have
$
  \mathbf{1}^{\mathsf{T}}A\mathbf{1}
  =\mathbf{1}^{\mathsf{T}}(\abs{I}d^{-1/2}\mathbf{1} + \zeta)
  =\abs{I}d^{-1/2}\norm{\mathbf{1}}^{2},
$
which shows that 
\begin{equation}
\label{eq:lambdaonelowerbound}
\lambda_1(A) \geq \abs{I}d^{-1/2}.
\end{equation}

\myparagraph{An upper bound on $\lambda_{2}(A)$}
We need the following bound which states what with sufficient probability
\begin{equation}
\label{eq:lambdatwobound}
\lambda_2(A) \lesssim |I|d^{-1}.
\end{equation}
  By the Courant-Fischer theorem we have
  \[
  \lambda_{2}(A) = \inf\limits_{v}\sup\limits_{x \in S^{d - 1}, \langle x, v\rangle = 0}x^{\mathsf{T}}Ax \le \sup\limits_{x \in S^{d - 1}, \langle x, \mathbf{1}\rangle = 0}x^{\mathsf{T}}Ax.
  \]
  Consider the $d \times d$ partial isometry matrix $R$ defined as follows. Fix an orthonormal basis $w_{1}, \ldots, w_{d}$ in $\mathbb{R}^{d}$ such that $w_{1}$ is proportional to $\mathbf{1}$. The matrix $R$ has its first row equal to zero and its $i$-th row for $i \ge 2$ equal to $w_i$. Observe that $R\; \mathbf{1} = 0$ and for any $v$ such that $\langle u, \mathbf{1}\rangle = 0$ we have $\|Ru\| = \|u\|$ together with $RR^{\mathsf{T}} = I_{d} - e_1e_1^{\mathsf{T}}$. Next, we show
  \begin{equation}
  \label{eq:detailed}
  \sup\limits_{x \in S^{d - 1}, \langle x, \mathbf{1}\rangle = 0}x^{\mathsf{T}}Ax = \sup\limits_{x \in S^{d - 1}}x^{\mathsf{T}}RAR^{\mathsf{T}}x.
  \end{equation}
  Indeed, consider a maximizer $x_0 \in S^{d - 1}$ of the right-hand side. We have that $R^{\mathsf{T}}x_0$ is orthogonal to $\mathbf{1}$ since $\mathbf{1}^{\mathsf{T}}R^{\mathsf{T}}x_0 = (R\; \mathbf{1})^{\mathsf{T}}x_0 = 0$. Finally, we have that for any $x^\prime \in S^{d - 1}$ such that $\langle x^{\prime}, \mathbf{1}\rangle = 0$ there is $x \in S^{d - 1}$ such that $R^{\mathsf{T}}x = x^{\prime}$. This is because $x^{\prime} = \alpha_{2}w_2 + \ldots \alpha_d w_d = R^{\mathsf{T}} x$, where $x^{\mathsf{T}} = (0, \alpha_2, \ldots, \alpha_d) \in S^{d - 1}$. Therefore, \eqref{eq:detailed} follows.

Further, the matrix $RAR^{\mathsf{T}}$ is non-negative semi-definite as well as each
additive term that forms it. We have
\[
RAR^{\mathsf{T}} = \sum_{i \in I} RX_{i}X_{i}^{\mathsf{T}}R^{\mathsf{T}}
\]
and for the operator norm we have
$\|RX_{i}X_{i}^{\mathsf{T}}R^{\mathsf{T}}\| \le \|R\|\|X_iX_i^{\mathsf{T}}\| \|R^{\mathsf{T}}\| =
\norm{R}\norm{R^{\mathsf{T}}}\le 1$.
Note that
\begin{equation}
\E RAR^{\mathsf{T}} = R\E \left(A\right) R^{\mathsf{T}}
= \abs{I}R\left((d^{-1} - (d^{3/2} + d)^{-1})I_{d} + (d^{3/2} +
  d)^{-1}\mathbf{1}\mathbf{1}^{\mathsf{T}}
  \right)R^{\mathsf{T}}.
\end{equation}
Using $R\,\mathbf{1} = 0$, the above simplifies to
\begin{equation}
\E RAR^{\mathsf{T}} = R\E \left(A\right) R^{\mathsf{T}}
= \abs{I}(d^{-1} - (d^{3/2} + d)^{-1})RR^{\mathsf{T}}.
\end{equation}
Since $RR^{\mathsf{T}} = I_{d} - e_{1}e_{1}^{\mathsf{T}}$, we have
$\lambda_{1}(\E RAR^{\mathsf{T}}) = \abs{I}(d^{-1} - (d^{3/2} + d)^{-1})$.
Applying the matrix Chernoff inequality \citep[Theorem 5.1.1]{tropp15}
we obtain
\begin{align}
\Pr\left(
  \lambda_{2}(A) \geq 2\abs{I}(d^{-1} - (d^{3/2} + d)^{-1})
\right)
&\leq
\Pr\left(
  \lambda_{1}(RAR^{\mathsf{T}}) \geq 2\abs{I}(d^{-1} - (d^{3/2} + d)^{-1})
\right)
\\
&\leq
d(e/4)^{\abs{I}(d^{-1} + (d^{3/2} + d)^{-1})}.
\label{eq:upperboundlambdatwo}
\end{align}
The above probability is arbitrary small for large enough $d$ and $n \gtrsim d^{3}\log d$. The bound follows.

\myparagraph{A lower bound on $\zeta^{\mathsf{T}}A^{-1}\zeta$} Let $u_1, \ldots, u_d$ be an orthonormal basis of eigenvectors of $A$. Using the spectral decomposition and $\sum_{i=1}^{d}\ip{u_{i}}{\zeta}^{2} = \|\zeta\|^{2}$, we have
\begin{equation}
  \label{eq:zeta-quadratic-form-lower-bound}
  \zeta^{\mathsf{T}}A^{-1}\zeta
  = \sum_{i=1}^{d} \ip{u_i}{\zeta}^{2}\lambda_i(A)^{-1}
  \geq \lambda_2(A)^{-1}\sum_{i=2}^{d} \ip{u_i}{\zeta}^{2}
  = \lambda_2(A)^{-1}(\norm{\zeta}^{2} - \ip{u_1}{\zeta}^{2}).
\end{equation}
By~\eqref{eq:lambdatwobound} we have
$\lambda_2(A) \lesssim nd^{-3/2}$ and by above computations $\norm{\zeta}^2 \sim n$. Therefore, the claim immediately follows if we prove that $\ip{u_1}{\zeta}^{2} \ll n$. Note that
\[
  \lambda_1(A)\ip{u_1}{\zeta}
  =\ip{Au_1}{\zeta}
  =\ip{(A - \E A)u_1}{\zeta} + \ip{(\E A)u_1}{\zeta}
 \]
  implies
  \begin{equation}
  \label{eq:u1-zeta-upper-bound}
   \ip{u_1}{\zeta}^{2}
  \le \lambda_1(A)^{-2}
    (\norm{A - \E A}\norm{\zeta} + |\ip{(\E A) u_1}{\zeta}|)^{2}.
  \end{equation}
Recall that
\[
  \E A = |I|\left((d^{-1} - (d^{3/2} + d)^{-1})I_{d} + (d^{3/2} + d)^{-1}\mathbf{1}\mathbf{1}^{\mathsf{T}}\right)
  \quad\text{and}\quad
  \ip{\mathbf{1}}{\zeta} = 0.
\]
We have
\[
  |\ip{(\E A) u_{1}}{\zeta}|
  = \abs{I}(d^{-1} + (d^{3/2} + d)^{-1})|\ip{u_1}{\zeta}| \le \abs{I}(d^{-1} + (d^{3/2} + d)^{-1})\|\zeta\|.
\]
Using that $(X_{i}X_{i}^{\mathsf{T}})^{2} = X_{i}X_{i}^{\mathsf{T}}$ for $i \in I$, we have
\[
\left\|\sum_{i=1}^{\abs{I}} \E\left(X_{i}X_{i}^{\mathsf{T}} -
  \E X_{i}X_{i}^{\mathsf{T}}\right)^{2}\right\| \le \left\|\sum_{i=1}^{\abs{I}} \E\left(X_{i}X_{i}^{\mathsf{T}}\right)^{2}\right\| = \norm{\E A} \le 2\abs{I}d^{-1/2}.
\]
Applying the matrix Bernstein inequality
\citep[Theorem 6.6.1]{tropp15} we obtain
\begin{equation}
\label{eq:matrixbernstein}
  \Pr\left(\norm{A - \E A} \geq \abs{I}d^{-1}\right)
  \leq d \exp\left(-\frac{\abs{I}^{2}d^{-2}/2}{2\abs{I}d^{-1/2} + d/3}\right),
\end{equation}
where the above probability is arbitrary small for large enough $d$ and $n \gtrsim d^{3}\log d$.
Hence,~\eqref{eq:u1-zeta-upper-bound} gives with sufficient probability
$$
  \ip{u_1}{\zeta}^{2}
  \leq 2\lambda_1(A)^{-2}(\norm{A - \E A}^{2} + 2\abs{I}^{2}d^{-2})\norm{\zeta}^{2}
  \lesssim \lambda_1(A)^{-2}\abs{I}^{2}d^{-2}\norm{\zeta}^{2} \lesssim  n/d.
$$
The claim follows.

\myparagraph{An upper bound on $\mathbf{1}^{\mathsf{T}}A^{-1}\mathbf{1}$} As before let $u_1, \ldots, u_d$ be an orthonormal basis of eigenvectors of $A$.
Using the lower bound \eqref{eq:lambdaonelowerbound} and $\sum_{i=2}^{d}\ip{u_{i}}{\mathbf{1}}^{2} = d - \ip{u_{1}}{\mathbf{1}}^{2}$, we have
\begin{align}
  \label{eq:quadratic-form-general-expression}
  \mathbf{1}^{\mathsf{T}}A^{-1}\mathbf{1}
  &=
  \sum_{i=1}^{d}\ip{u_{i}}{\mathbf{1}}^{2}\lambda_i(A^{-1}) \nonumber
  \\
  &\le \ip{u_1}{\mathbf{1}}^{2}/\lambda_1(A)
    + (d - \ip{u_1}{\mathbf{1}}^{2})/\lambda_d(A) \le d^{3/2}/|I|+ (d - \ip{u_1}{\mathbf{1}}^{2})/\lambda_d(A).
\end{align}
We want to provide an upper bound on $d - \ip{u_1}{\mathbf{1}}^{2}$. By \eqref{eq:upperboundlambdatwo} we have that with sufficient probability $\lambda_{j}(A) \le 2|I|d^{-1}$ for $j = 2, \ldots, d$ and therefore, for the same values of $j$ we have $\lambda_{j}(A)/\lambda_1(A) \le 2d^{-1/2}$. Using the last fact we have
\[
\frac{\mathbf{1}^{\mathsf{T}}A\mathbf{1}}{\lambda_1(A)}
  = 
  \ip{u_{1}}{\mathbf{1}}^{2} + \sum_{i=2}^{d}\ip{u_{i}}{\mathbf{1}}^{2}\frac{\lambda_i(A)}{\lambda_1(A)}
  \le \ip{u_1}{\mathbf{1}}^{2}+  2d^{-1/2}(d - \ip{u_1}{\mathbf{1}}^{2}).
\]
The last inequality combined with the fact that $\mathbf{1}$ is the first eigenvector of $\E A$ implies for $d \ge 16$,
\begin{equation}
\label{eq:boundondiff}
d - \ip{u_1}{\mathbf{1}}^{2} \le (1 - 2d^{-1/2})^{-1}\left(d - \frac{\mathbf{1}^{\mathsf{T}}A\mathbf{1}}{\lambda_1(A)}\right) \le 2\left(\frac{\mathbf{1}^{\mathsf{T}}\E A\mathbf{1}}{\lambda_1(\E A)} - \frac{\mathbf{1}^{\mathsf{T}}A\mathbf{1}}{\lambda_1(A)}\right)
\end{equation}
Finally, since $|\lambda_{1}(A) - \lambda_{1}(\E A)| \le \|A - \E A\|$ and by the lower bound \eqref{eq:lambdaonelowerbound} we have
\begin{align}
\frac{\mathbf{1}^{\mathsf{T}}\E A\mathbf{1}}{\lambda_1(\E A)} - \frac{\mathbf{1}^{\mathsf{T}}A\mathbf{1}}{\lambda_1(A)} &= \mathbf{1}^{\mathsf{T}}\E A\mathbf{1}\left(\frac{1}{\lambda_1(\E A)} - \frac{1}{\lambda_1(A)}\right) + \frac{\mathbf{1}^{\mathsf{T}}(\E A - A)\mathbf{1}}{\lambda_1(A)}
\\
&\le \frac{2d\|A - \E A\|}{\lambda_{1}(A)} \le \frac{2d^{3/2}\|A - \E A\|}{|I|}.\label{eq:boundondifffinal}
\end{align}
By \eqref{eq:matrixbernstein}, we have that with sufficient probability $\|A - \E A\| \lesssim |I|d^{-1}$. Therefore, combining this with \eqref{eq:boundondiff} we have with sufficient probability 
\[
d - \ip{u_1}{\mathbf{1}}^{2}  \lesssim \sqrt{d}.
\]
Plugging the above inequality into \eqref{eq:quadratic-form-general-expression} and using our lower bound $\lambda_{d}(A) \gtrsim |I|d^{-1}$ we prove the claim.
\myparagraph{An upper bound on $\mathbf{1}^{\mathsf{T}}A^{-2}\mathbf{1}$}
The proof is completely analogous to the case $\mathbf{1}^{\mathsf{T}}A^{-1}\mathbf{1}$. We have
\begin{align*}
  \mathbf{1}^{\mathsf{T}}A^{-2}\mathbf{1}
  &\le \ip{u_1}{\mathbf{1}}^{2}/(\lambda_1(A))^2
    + (d - \ip{u_1}{\mathbf{1}}^{2})/(\lambda_d(A))^2
    \\
    &\le d^{2}/|I|^2+ (d - \ip{u_1}{\mathbf{1}}^{2})/(\lambda_d(A))^2.
\end{align*}
As before, with sufficient probability we have $d^{2}/|I|^2 \lesssim n^{-2}d^3$. The only difficulty is that we need a slightly sharper variant of the upper bound on $\|A - \E A\|$. Recalling the bound \eqref{eq:matrixbernstein}, by the matrix Bernstein inequality \citep[Theorem 6.6.1]{tropp15} we have
\[
 \Pr\left(\norm{A - \E A} \geq \abs{I}d^{-3/2}\right)
 \leq d \exp\left(-\frac{\abs{I}^{2}d^{-3}/2}{4\abs{I}d^{-1/2} + d/3}\right),
\]
which is arbitrarily small provided that $d$ is large enough and $n \gtrsim d^{3}\log d$. Observe that this is the step where we have our strongest requirement on $n$. Note that using that by matrix Chernoff inequality, as shown above in the proof that $A$ is invertible, we have with sufficient probability:
$$
  \lambda_{d}(A) \gtrsim \abs{I}d^{-1}.
$$
Using \eqref{eq:boundondiff}, \eqref{eq:boundondifffinal} and the two inequalities above, we conclude that the following holds with sufficient probability:
\[
    \frac{d - \ip{u_1}{\mathbf{1}}^{2}}{\lambda_d(A)^2}
    \lesssim \frac{d^{3/2}\|A - \E A\| / \abs{I}}{\abs{I}^{2}d^{-2}}
    \lesssim \frac{d^{7/2}\abs{I}d^{-3/2} }{\abs{I}^{3}}
    = \frac{d^{2}}{\abs{I}^{2}}
    \lesssim \frac{d^{3}}{n^{2}}.
\]
The proof of our result is complete.
 \qed

\section{Proof of the Theorem by Forster and Warmuth}
\label{section:proof-of-optimal-bound}

For the sake of completeness, in this section we present the leave-one-out analysis due to \citep{forster2002relative}. Following the notation used in the proof of Theorem~\ref{thm:main_thm_second}, we introduce an additional independent element $(X_{n + 1}, Y_{n + 1})$ distributed according to $P$ to the sample $S_{n+1}$. With a slight abuse of notation, we define for $j = 1, \ldots, n + 1$ the unnormalized and unpenalized (as opposed to \eqref{eq:samplecov}) empirical second moment matrices by
\[
\tilde{\Sigma} = \sum\limits_{i = 1}^{n + 1}X_iX_i^{\mathsf{T}} \quad
\text{and} \quad \tilde{\Sigma}^{(-j)} = \sum\limits_{i = 1, i \neq j}^{n + 1}X_iX_i^{\mathsf{T}}.
\]
Given the sample $S_{n+1}$, fix the minimum $\ell_{2}$ norm ERM defined as follows:
$$
  \what_{\infty} = \tilde{\Sigma}^{\dagger}\Big(\sum_{i=1}^{n+1}Y_{i}X_{i}\Big),
$$
where recall that $\tilde{\Sigma}^{\dagger}$ denotes the Moore-Penrose inverse of the matrix $\tilde{\Sigma}$. For $i=1,\dots, n+1$, let
$h_{i}$ denote the $i$-th leverage score:
$$
    h_{i} = X_{i}^{\mathsf{T}}\tilde{\Sigma}^{\dagger}X_{i}.
$$
Let $\widehat{f}^{(-j)}$ denote the non-linear estimator trained on the $n$ samples $S_{n+1} \setminus \{ (X_{j}, Y_{j}) \}$. By the definition of $\widehat{f}^{(-j)}$, we have
\begin{align}
    \widehat{f}^{(-j)}(X_{j})
    &=
    (1 - h_{j})
    \left\langle \left(\tilde{\Sigma}^{(-j)} + X_{j}X_{j}^{\mathsf{T}}\right)^{\dagger}\left(\sum_{i=1, i\neq     j}^{n+1}Y_{i}X_{i}\right), X_{j}
    \right\rangle \\
    &= (1-h_{j})\left\langle\tilde{\Sigma}^{\dagger}\left(\sum_{i=1}^{n+1}Y_{i}X_{i}\right) - \tilde{\Sigma}^{\dagger}Y_{j}X_{j}, X_{j} \right\rangle \\
    &= (1-h_{j})(\langle \what_{\infty}, X_{j} \rangle - h_{j}Y_{j}). \label{eq:non-linear-predictor-prediction}
\end{align}

The analysis in \citep{forster2002relative} presented below is akin to the one used in our ridge regression proof (cf.\ Appendix~\ref{sec:secondmainthm}), albeit with one simplifying modification. Instead of decomposing the excess risk into the stability and fitting terms, we use the following \emph{leave-one-out} decomposition:
\begin{align}
    &\E_{S_{n+1}} \left(\widehat{f}^{(-(n+1))}(X_{n+1}) - Y_{n+1}\right)^{2}
      - \inf_{\omega \in \R^{d}} R(\omega) \\
    &= \E_{S_{n+1}} \left( \frac{1}{n+1}\sum_{i=1}^{n+1} \left(\widehat{f}^{(-i)}(X_{i}) - Y_{i}\right)^{2} \right)
      - \inf_{\omega \in \R^{d}} \E_{S_{n+1}} \left(\frac{1}{n+1} \sum_{i=1}^{n} \left(\langle \omega, X_{i} \rangle - Y_{i}\right)^{2}\right) \\
    &\leq \E_{S_{n+1}} \left( \frac{1}{n+1}\sum_{i=1}^{n+1} \left(\widehat{f}^{(-i)}(X_{i}) - Y_{i}\right)^{2} \right)
      - \E_{S_{n+1}} \inf_{\omega \in \R^{d}} \left(\frac{1}{n+1} \sum_{i=1}^{n} \left(\langle \omega, X_{i} \rangle - Y_{i}\right)^{2}\right) \\
    &= \E_{S_{n+1}} \left( \frac{1}{n+1}\sum_{i=1}^{n+1} \left(\widehat{f}^{(-i)}(X_{i}) - Y_{i}\right)^{2} \right)
      - \E_{S_{n+1}} \left(\frac{1}{n+1} \sum_{i=1}^{n} \left(\langle \what_{\infty}, X_{i} \rangle - Y_{i}\right)^{2}\right) \\
    &= \E_{S_{n+1}} \left( \frac{1}{n+1}\sum_{i=1}^{n+1} \left(\widehat{f}^{(-i)}(X_{i}) - Y_{i}\right)^{2} 
      - \left(\langle \what_{\infty}, X_{i} \rangle - Y_{i}\right)^{2} \right) \label{eq:loo-decomposition}.
\end{align}
Plugging in \eqref{eq:non-linear-predictor-prediction} into the above summands, we obtain
\begin{align*}
    &\left(\widehat{f}^{(-i)}(X_{i}) - Y_{i}\right)^{2}
      - \left(\langle \what_{\infty}, X_{i} \rangle - Y_{i}\right)^{2} \\
    &=\left((1 - h_{j})\langle \what_{\infty}, X_{j} \rangle - (1 + h_{j} - h_{j}^{2}) Y_{i}\right)^{2}
      - \left(\langle \what_{\infty}, X_{i} \rangle - Y_{i}\right)^{2} \\
    &= \left((1 - h_{j})^{2} - 1 \right)\langle \what_{\infty}, X_{j} \rangle^{2}
     - 2\left((1 - h_{j})(1 + h_{j} - h_{j}^{2}) - 1 \right)\langle \what_{\infty}, X_{j} \rangle Y_{j}
     \\
     &\quad\quad+ ((1 + h_{j} - h_{j}^{2})^{2} - 1)Y_{j}^{2}.
\end{align*}
If $h_{j} = 0$, then the above expression is equal to zero. Assume that $h_{j} > 0$ (hence, $h_{j} \in (0, 1]$).
Then, the coefficient preceding $\langle \what_{\infty}, X_{j} \rangle^{2}$ is negative; optimizing the quadratic equation we have
\begin{align}
    &\left((1 - h_{j})^{2} - 1 \right)\langle \what_{\infty}, X_{j} \rangle^{2}
     - 2\left((1 - h_{j})(1 + h_{j} - h_{j}^{2}) - 1 \right)\langle \what_{\infty}, X_{j} \rangle Y_{j}
      \\
      &\quad\quad+ ((1 + h_{j} - h_{j}^{2})^{2} - 1)Y_{j}^{2}\\
    &\leq \frac{\left((1 - h_{j})(1 + h_{j} - h_{j}^{2}) - 1 \right)^{2}
    }{1 - (1 - h_{j})^{2}}Y_{j}^{2}
    + ((1 + h_{j} - h_{j}^{2})^{2} - 1)Y_{j}^{2} \\
    &= (2h_{j} - h_{j}^{2})Y_{j} \leq 2h_{j}Y_{j}^2.
\end{align}
Plugging the above into \eqref{eq:loo-decomposition} yields:
\begin{align}
    &\E_{S_{n+1}} \left(\widehat{f}^{(-(n+1))}(X_{n+1}) - Y_{n+1}\right)^{2}
      - \inf_{\omega \in \R^{d}} R(\omega) \\
    &\leq \E_{S_{n+1}} \left( \frac{1}{n+1}\sum_{i=1}^{n+1} \left(\widehat{f}^{(-i)}(X_{i}) - Y_{i}\right)^{2} 
      - \left(\langle \what_{\infty}, X_{i} \rangle - Y_{i}\right)^{2} \right) \\
    &\leq \E_{S_{n+1}} \frac{1}{n+1} \sum_{i=1}^{n+1}2h_{i}Y_{i}^{2} \leq \frac{2dm^{2}}{n+1},
\end{align}
where in the last line we used the fact that $\sum\limits_{i = 1}^{n + 1}h_{i} = \operatorname{rank}(\tilde{\Sigma}) \leq d$. The proof is complete. \hfill\qed
\end{appendix}

\end{document}